\documentclass[12pt]{article}
\usepackage{amssymb,amsmath,amsthm,amsfonts}
\usepackage{mathrsfs}
\usepackage{graphicx}
\usepackage{epsfig}
\usepackage{tikz}
\usepackage{mathtools}

\def\disp{\displaystyle}

\theoremstyle{plain}
\newtheorem{theorem}{Theorem}[section]
\newtheorem{lemma}{Lemma}[section]

\theoremstyle{definition}

\numberwithin{equation}{section}

\begin{document}
\title{\large\bf Asymptotic behavior of a doubly haptotactic cross-diffusion model for oncolytic virotherapy}

\author{
{\rm  Yifu Wang$^{1,*}$, Chi Xu$^{1}$
}\\[0.2cm]
{\it\small \rm School of Mathematics and Statistics,  Beijing Institute of Technology}\\
{\it\small \rm Beijing 100081, P.R. China$^1$}
}
\date{}

\maketitle
\renewcommand{\thefootnote}{\fnsymbol{footnote}}
\setcounter{footnote}{-1}
\footnote{$^{*}$Corresponding author. E-mail addresses:wangyifu@bit.edu.cn (Y.Wang), XuChi1993@126.com (C. Xu)}

\maketitle
\begin{abstract}
This paper considers a model for oncolytic virotherapy given by the doubly haptotactic cross-diffusion system
\begin{equation*}
 \left\{\begin{array}{ll}
  u_t=D_u\Delta  u-\xi_u\nabla\cdot(u\nabla v)+\mu_u u(1-u)-\rho uz,\\
 v_t=- (\alpha_u u+\alpha_w w)v,\\
w_t=D_w\Delta  w-\xi_w\nabla\cdot(w\nabla v)- w+\rho uz,\\
z_t=D_z\Delta z-\delta_z z- \rho uz+\beta w,
 \end{array}\right.
\end{equation*}
with positive parameters  $D_u,D_w,D_z,\xi_u,\xi_w,\delta_z,\rho$,
$\alpha_u,\alpha_w,\mu_u,\beta$.
 When
posed under no-flux boundary conditions in a smoothly bounded domain $\Omega\subset {\mathbb{R}}^2$, and along with  initial conditions involving suitably regular data, the global existence of classical solution to this system was asserted in  Tao and Winkler (2020).  Based on the suitable quasi-Lyapunov functional,  it is shown that when the virus replication rate $\beta<1$, the global classical solution $(u,v,w,z)$ is uniformly bounded and exponentially stabilizes to
  the constant
equilibrium $(1, 0, 0, 0)$ in the topology $(L^\infty(\Omega))^4 $  as $t\rightarrow \infty$.

\end{abstract}
\vspace{0.3cm}


\section{Introduction}
As compared to the traditional treatment like chemotherapy or radiotherapy for cancer diseases, the prominent advantage of virotherapy is that  the therapy can reduce the side-effect on the healthy tissue. In clinical treatments, the so-call
oncolytic viruses (OV) which are either genetically engineered or naturally occurring can selectively attack the cancer cells and eventually destroy them
without harming normal cells because virus can replicate inside the infected cells  and  proceed to infect adjacent cancer cells
 with the aim to drive the tumor cells to extinction  (\cite{FIT,GPKLK}).
 Despite some partial success, implementation of virotherapy is not in sight. In fact,
 clinical data reveal that the efficacy of virotherapy will be reduced by many factors, such as
  circulating antibodies,~various immune cells or even deposits of extracellular matrix may essentially decrease (\cite{GK,NG}). Therefore, to facilitate the understanding of the mechanisms that hinder virus spread,  the authors of \cite{ARD-MB} proposed a mathematical model to describe  the interaction between both uninfected and infected cancer cells, as well as extracellular matrix (ECM) and oncolytic virus particles,   which is given by
 \begin{equation}\label{1.1}
 \left\{
 \begin{array}{lll}
 u_{t}=D_{u}\Delta u-\xi_u\nabla \cdot(u\nabla v)+\mu_u u(1-u)-\rho_u uz,&&x\in\Omega,~t>0,\\
 v_{t}=-(\alpha_u u+\alpha_w w)v	+\mu_v v(1-v),&&x\in\Omega,~t>0,\\
 w_{t}=D_w\Delta w-\xi_w\nabla \cdot(w\nabla v)-\delta_w w +\rho_w uz,&&x\in\Omega,~t>0,\\
 z_{t}=D_z\Delta z-\delta_z z-\rho_z uz+\beta w,&&x\in\Omega,~t>0,
 \end{array}
 \right.	
 \end{equation}
in  a   smoothly bounded domain  $\Omega \subset \mathbb{R}^n$,  with positive parameters  $D_u,D_w,D_z,\xi_u,\xi_w,\alpha_u,\alpha_w$, $\mu_u,  \delta_w,\delta_z,\beta$ and nonnegative constants $\mu_v,\rho_u, \rho_w, \rho_z$, and with the
unknown variables $u, w, z$ and $v$ denoting the population densities of uninfected cancer cells, infected
cancer cells, virus particles and ECM, respectively. Here
the crucial modeling hypothesis underlying \eqref{1.1}, which accounts for haptotactic motion of cancer cells and thereby marks
a substantial difference between \eqref{1.1} and related more classical reaction-diffusion models for virus
dynamics (\cite{K,PZS}), is that apart from its random diffusion, both uninfected and infected cancer cells bias their motion
upward ECM gradients  simultaneously due to the attraction by some macromolecules trapped in the ECM.
In addition, the oncolytic virus particles  infect the uninfected cancer cells upon contact with uninfected tumour cells,  and  new infectious virus particles  are released  at rate $\beta>0$ when infected
cells burst (a process known as lysis); beyond this, \eqref{1.1} presupposes that the ECM is degraded upon interacting with both type of cancer cells, and is possibly remodeled by the normal tissue according to logistic laws.

Due to its relevance in several biological contexts, inter alia the cancer invasion
(\cite{ACNST,CL}),  haptotaxis mechanism has received considerable attention in the analytical literature
(\cite{Cao,Jin,JinTian, LiL,PW3MAS,SSW, TaoSiam, TW-JDE, WalkerW,WYF,ZCU,ZhengK}).
The most characteristic ingredient of the model \eqref{1.1}  is the presence of two simultaneous haptotaxis processes of cancer cells,
and thereby distinguishes  it from most haptotaxis (\cite{Jin,WalkerW,ZCU}) and chemotaxis-haptotaxis systems (\cite{Cao,PW3MAS,TaoSiam}) studied in the literature, especially the ECM is degraded  by both type of cancer cells in \eqref{1.1} directly,  rather  matrix degrading enzymes (MDEs) secreted by tumor cells (see \cite{JinTian,PW3MAS,PW3MAS2, TaoSiam} for example). It is observed that the former circumstance seems to widely restrict the accessibility  to the approaches well established in the analysis of  related reaction--diffusion systems, and accordingly the considerable challenges arise for the rigorous  analysis of \eqref{1.1}, particularly when addressing issues
related to qualitative solution behaviour.

To the best of our knowledge, so far the quantitative comprehension available for \eqref{1.1} is yet mainly limited in some simple setting
(\cite{Chen,LiWang, RL-MMA, TW-PRSE,TW-DCDS,TW-JDE2020,TW-EJAM,TW-NA}). For instance, based on the construction of certain quasi--Lyapunov functional, Tao and Winkler (\cite{TW-JDE2020}) established  the global classical solvability  of \eqref{1.1}
in the two-dimensional case. With respect to the boundedness of solutions to \eqref{1.1}, authors in \cite{LiWang} considered  some slightly more comprehensive variants of \eqref{1.1}, which accounts for the haptotaxis mechanisms of both cancer cells and virions, in the situation when zero-order term  has suitably strong degradation. Apart from that, existing analytical works  indicate that the virus reproduction rate relative to the lysis rate of infected cancer cells
appears  to be critical
in determining the large time behavior of the corresponding solutions at least in some simplified version of  \eqref{1.1}, inter alia upon neglecting
haptotactic cross-diffusion of infected cancer cells and renew of ECM. Indeed, for the reaction--diffusion--taxis system
\begin{equation}\label{1.2}
 \left\{
 \begin{array}{lll}
 u_{t}=D_{u}\Delta u-\xi_u\nabla \cdot(u\nabla v)+\mu_u u(1-u)-\rho uz,&&x\in\Omega,~t>0,\\
 v_{t}=-(\alpha_u u+\alpha_w w)v,&&x\in\Omega,~t>0,\\
 w_{t}=D_w\Delta w-w +uz,&&x\in\Omega,~t>0,\\
 z_{t}=D_z\Delta z-z-uz+\beta w,&&x\in\Omega,~t>0,
 \end{array}
 \right.	
 \end{equation}
 it is shown in \cite{TW-DCDS} that if $\beta>1$, then for any reasonably regular initial data  satisfying $\overline{u_0}>\frac{1}{\beta-1}$
  the global classical solution of \eqref{1.2} with $\rho=0,\mu_u=0$ must  blow up in infinite time, which is also implemented by the result on boundedness in the case when  $\overline{u_0}<\frac{1}{(\beta-1)_{+}}$ and $v_0\equiv 0$ for any $\beta>0$. Beyond the latter,  it was proved that
  when $\rho>0$ and $\mu_u=0$,  the first solution component $u$ of \eqref{1.2} possesses a positive lower bounds whenever $0<\beta<1$ and  the initial data $u_0\not\equiv 0$ (\cite{TW-EJAM}).
Furthermore, as a extension of above outcome, the asymptotic behavior of solution  was investigated in \cite{TW-NA} if $0<\beta<1$. It is remarked that for  system \eqref{1.2} with $\mu_u>0$ and $0<\beta<1$,   the convergence properties of the corresponding solutions was also discussed in \cite{Chen}. We would like to mention
that as the complementing  results of  \cite{TW-NA},     the recent paper \cite{TW-PRSE} reveals that for any  prescribed level
$\gamma\in (0,\frac{1}{(\beta-1)_{+}})$, the corresponding solution of \eqref{1.2} with $\mu_u=0,\rho\geq 0,\beta>0$
will approach the constant equilibrium $(u_\infty, 0, 0, 0)$ asymptotically with some $u_\infty>0$  whenever
the initial deviation from homogeneous distribution $(\gamma, 0, 0, 0)$ is suitably small.

The purpose of this work is to investigate the dynamical features of the models involving the simultaneous haptotactic  processes
 of both uninfected and infected  cancer cells when the virus replication rate $\beta<1$.
 To this end, we are concerned with the comprehensive haptotactic cross-diffusion systems of the form
 \begin{equation}\label{1.3}
 \left\{\begin{array}{ll}
  u_t=D_u\Delta  u-\xi_u\nabla\cdot(u\nabla v)+\mu_u u(1-u)-\displaystyle\rho uz,&
x\in \Omega, t>0,\\
\disp{w_t=D_w\Delta  w-\xi_w\nabla\cdot(w\nabla v)-  w+\displaystyle\rho uz}, &
x\in \Omega, t>0,\\
\disp{ v_t=- (\alpha_u u+\alpha_w w)v},&
x\in \Omega, t>0,\\
\disp{z_t=D_z\Delta z-\delta_z z-\displaystyle\rho uz+\beta w}, &
x\in \Omega, t>0,\\
\disp{(D_u\nabla  u-\xi_u u\nabla v)\cdot \nu=(D_w\nabla  w-\xi_w w\nabla v)\cdot \nu=\nabla z\cdot \nu=0},&
x\in \partial\Omega, t>0,\\
\disp{u(x,0)=u_0(x)},w(x,0)=w_0(x), v(x,0)=v_0(x),z(x,0)=z_0(x), & x\in \Omega
 \end{array}\right.
\end{equation}
in a smoothly bounded domain $\Omega\subset \mathbb{R}^2$. To import the precise framework underlying the basic theory from
\cite{TW-JDE2020} we shall henceforth assume that
 \begin{equation}\label{1.4}
\left\{
\begin{array}{ll}
\displaystyle{u_0,w_0,z_0~ \hbox{and}\, v_0 ~\hbox{are  nonnegative functions from}~C^{2+\vartheta}(\bar{\Omega})~
\hbox{for some}~\vartheta\in (0,1),} \\
\displaystyle{\mbox{with}~u_0\not\equiv0,~w_0\not\equiv0,~z_0\not\equiv0,~v_0\not\equiv0~\hbox{and}~\frac{\partial w_0}{\partial\nu}=0~~\mbox{on}~~\partial\Omega.} \\
\end{array}
\right.
\end{equation}
 Hence the outcome of \cite{TW-JDE2020} asserts the global existence of a unique classical solution $(u,v,w,z)$ to \eqref{1.3}.
 Our main results reveal that whenever $\beta<1$,  $(u,v,w,z)$ is uniformly bounded and exponentially
 converges  to  the constant equilibrium $(1, 0, 0, 0)$ in the topology $(L^\infty(\Omega))^4$ in a large
time limit, which can be stated as follows
\begin{theorem}\label{Theorem 1.1}
Let $\Omega \subset \mathbb{R}^2$ be a bounded domain with smooth boundary,~$D_u,D_w,D_z,\xi_u,\xi_w$,
$\mu_u,\rho,\alpha_u,\alpha_w,\delta_z$ are positive parameters, and suppose that $0<\beta<1$. Then system \eqref{1.3} admits a unique global classic positive solution satisfying
\begin{equation}\label{1.5}
\sup\limits_{t>0}\left\{\|u(\cdot,t)\|_{L^{\infty}(\Omega)}+\|v(\cdot,t)\|_{L^{\infty}(\Omega)}+\|w(\cdot,t)\|_{L^{\infty}(\Omega)}+
\|z(\cdot,t)\|_{L^{\infty}(\Omega)}\right\}<\infty.
\end{equation}
Moreover there exist positive constants $\eta,\varrho,\gamma_1$, $\gamma_2$ and  $C>0$ such that
\begin{equation}\label{1.6}
\|u(\cdot,t)-1\|_{L^{\infty}(\Omega)}\leq Ce^{-\eta t},	
\end{equation}
\begin{equation}\label{1.7}
\|w(\cdot,t)\|_{L^{\infty}(\Omega)}\leq C e^{-\varrho t},	
\end{equation}
\begin{equation}\label{1.8}
\|z(\cdot,t)\|_{L^{\infty}(\Omega)}\leq Ce^{-\gamma_1t}	
\end{equation}
as well as
\begin{equation}\label{1.9}
\|v(\cdot,t)\|_{L^{\infty}(\Omega)}\leq Ce^{-\gamma_2 t}.
\end{equation}

\end{theorem}
Since the third equation in \eqref{1.3} is merely an ordinary differential equation, no smoothing action on the spatial
regularity of $v$ can be expected.  To overcome the analytical difficulties arising from the latter, inter alia in the derivation of global boundedness of solutions, we accordingly introduce the variable
transformation $a=e^{-\chi_u v}u$ and $b=e^{-\chi_w v}w$, and establish a priori estimate for the solution components $a,b$ of the corresponding equivalent system \eqref{2.2} below in the space $LlogL(\Omega)$ rather than the solution components $u,w$ to \eqref{1.3}. Note that
in the evolution of density $v$ of ECM fibers the quantity $u,w$ appears via a sink term, whereas it  turns to a genuine superlinear production terms of system \eqref{2.2} in  the style of $\chi_u a(\alpha_u ae^{\chi_u v}
+ \alpha_w be^{\chi_w v})v$ and $\chi_w b(\alpha_u ae^{\chi_u v}
+ \alpha_w be^{\chi_w v})v$.
Taking advantage of the exponential decay of $w$ in  $L^1$ norm in the case $0<\beta<1$, we shall track the time evolution of
$$
\mathcal{F}(t)=\int_{\Omega}e^{\chi_u v}a(\cdot,t)\log a(\cdot,t)+\int_{\Omega}e^{\chi_w v}b(\cdot,t)\log b(\cdot,t)+\int_{\Omega}z^2(\cdot,t)
$$
with $a=e^{-\chi_u v}u$ and $b=e^{-\chi_w v}w$, which is somewhat different from the quasi-Lyapunov functional (4.11) in \cite{TW-JDE2020} where a Dirichlet integral of $\sqrt{v}$ is involved. Here the quadratic degradation term in the first
equation of \eqref{1.3} seems to be necessary. Thereafter applying  a variant of the Gagliardo--Nirenberg inequality involving certain $LlogL$-type norms and performing a Moser-type iteration, the $L^{\infty}$-bounds of solutions is derived.

In addition, our result indicates that although haptotaxis mechanism may have some important influence on the properties of the related system on
short or intermediate time scales,  the large  time behavior of  solution to  \eqref{1.3}  can essentially be
described by the corresponding haptotaxis-free system at least under the biological meaningful restriction $\beta<1$. In order to prove Theorem 1.1, a first step is to derive a  pointwise lower bound for $a:=e^{-\chi_u v}u$ (Lemma \ref{Lemma42}), which, in turn, amounts to establishing an exponential decay of $z$ with respect to the norm in $L^\infty(\Omega)$ (Lemma \ref{Lemma42}). To achieve the latter, we will make use of  $L^1(\Omega)$-decay information of $w,z$ explicitly contained in Lemma \ref{Lemma23}. Scondly, as a consequence of the former, the  exponential decay of $v$  with respect to  $L^{\infty}(\Omega)$ norm is achieved (Lemma \ref{Lemma43}), which along a  $a^{-1}$--testing procedure will provide quite weak convergence information of $u$, inter alia the integrability property of $\nabla \sqrt{a}$ in $L^2((0,\infty); L^2(\Omega))$ (Lemma \ref{Lemma44}). The next step will consist of
verifying the integrability  of  $\nabla v$ in  $L^2((0,\infty); L^2(\Omega))$ rather than  that of $a_{t}$ (Lemma \ref{Lemma45}),  which  will  turn out to
be sufficient a condition  in the derivation of exponential decay property  of $\|u(\cdot,t)-1\|_{L^p(\Omega)}$. Indeed, this integrability property of  $\nabla v$ enables us to derive an exponential decay of
$\int_{\Omega}|\nabla v|^2$ (Lemma \ref{Lemma46}), upon which and through a testing procedure, it is shown
that the convergence property of $u$ actually takes place  
in the type of \eqref{4.22} (Lemma \ref{Lemma47}). Further, upon the above decay properties, we are able to verify that  $\int_{\Omega}|\nabla v|^4$   decays exponentially by means of the suitable quasi-Lyapunov functional (Lemma \ref{Lemma49}). At this position, thanks  to the integrability exponent  in $\int_{\Omega}|\nabla v|^4$
exceeding the considered spatial dimension $n=2$, the desired decay property stated in Theorem 1.1 can be exactly achieved.

This paper will be organized as follows: Section 2 will introduce an equivalent system of \eqref{1.3} and give out some basic priori estimates of classical solutions thereof, inter alia the weak decay properties of $w,z$. Section 3  will focus on the construction of an entropy-type functional, which entails certain $LlogL$-type norms and thereby  allows us to establish the $L^{\infty}$-bounds.  Finally, starting from the exponential decay of  quantities $w,z$ with respect to the norm in $L^1(\Omega)$, we established the exponential convergence properties  of the solutions in Section 4.

\section{Preliminaries}
Let us firstly recall the result in \cite {TW-JDE2020} which  warrants the global smooth solvability of problem \eqref{1.3}.
\begin{lemma}\label{lemma2.1} Let $\Omega \subset \mathbb{R}^2$ be a bounded domain with smooth boundary,~$D_u,D_w,D_z,\xi_u,\xi_w$,
$\mu_u,\rho,\alpha_u,\alpha_w,\delta_z,\beta$ are positive parameters. Then  for any choice of $(u_0,v_0,w_0,z_0)$ fulfilling \eqref{1.4}, the problem \eqref{1.3} \eqref{1.4} possesses a uniquely determined classical solution
 $(u,v,w,z)\in (C^{2,1}(\overline{\Omega} \times [0,\infty)))^4$
for which $u>0,w>0,z>0$ and $v\geq 0$.
\end{lemma}
Following the variable of change used in related literature
(\cite{FFH,LiWang,TW-NA,TW-JDE}), which can conveniently reformulate the haptotactic interaction in \eqref{1.3}, we  define
$\chi_u:=\frac{\xi_u}{D_u}$ and $\chi_w:=\frac{\xi_w}{D_w}$
and set
\begin{equation}\label{2.1}
a:=e^{-\chi_u v}u~~\hbox{and}~~b:=e^{-\chi_w v}w.
\end{equation}
Then we transform  \eqref{1.3} into an equivalent system as below
\begin{equation}
\label{2.2}
\left\{\begin{array}{ll}
a_t=D_u e^{-\chi_u v}\nabla\cdot(e^{\chi_u v}\nabla a)+f(a,b,v,c), & x\in\Omega,t>0,\\
b_t=D_w e^{-\chi_w v}\nabla\cdot(e^{\chi_w v}\nabla b)+g(a,b,v,c), &  x\in\Omega,t>0,\\
v_t=- (\alpha_u ae^{\chi_u v}
+ \alpha_w be^{\chi_w v})v, &   x\in\Omega,t>0,\\
z_t=D_z \Delta z-\delta_z z-\rho_z uz+\beta w, &  x\in\Omega,t>0,\\
\displaystyle\frac{ \partial a}{\partial\nu}=\frac{ \partial b}{\partial\nu}=\frac{ \partial z}{\partial\nu}=0,
  &   x\in\partial\Omega,t>0,\\
a(x,0)=u_0(x) e^{-\chi_u v_0(x)}, ~b(x,0)=w_0(x) e^{-\chi_w v_0(x)}, &  x\in \Omega,\\
 v_0(x,0)= v_0(x),~~ z(x,0)=z_0(x), &  x\in \Omega
  \end{array}
 \right.
 \end{equation}
with
$$
f(a,b,v,c):=\mu_u a(1-ae^{\chi_u v})-\displaystyle\rho az+\chi_u a(\alpha_u ae^{\chi_u v}
+ \alpha_w be^{\chi_w v})v,
$$
as well as
$$
g(a,b,v,c):=- b+\displaystyle\rho aze^{(\chi_u-\chi_w)v}
+
\chi_w b(\alpha_u ae^{\chi_u v}
+ \alpha_w be^{\chi_w v})v.
$$

In  our subsequent analysis, unless
otherwise stated we shall assume that $ (a, v, b, z)$ is the global classical solution to \eqref{2.2} addressed in Lemma 2.1.

The damping effects of quadratic degradation in the first equation in \eqref{1.3} will be important for us to verify  the global boundedness of the solutions.
Let we first apply straightforward argument to achieve the following basic $L^1$-bounds for $u,w$ and $z$, which is also valid for the solution components $a,b$ of \eqref{2.2}.
\begin{lemma}\label{Lemma22}
For all $t>0$,~the solution $(u,w,v,z)$ satisfies
\begin{equation}\label{2.3}
\int_{\Omega}u(\cdot,t)\leq \max\left\{\int_{\Omega}u_0,|\Omega|\right\}:= m_u,	
\end{equation}
and
\begin{equation}\label{2.4}
v\leq \|v_0\|_{L^{\infty}(\Omega)}:= m_v
\end{equation}
and
\begin{equation}\label{2.5}
\int_{\Omega}w(\cdot,t)\leq \max\left\{\|u_0\|_{L^1(\Omega)}+\|w_0\|_{L^1(\Omega)},\frac{|\Omega|\mu_u}{\min\{1,\mu_u\}}\right\}:= m_w,	
\end{equation}
as well as
\begin{equation}\label{2.6}
\int_{\Omega}z(\cdot,t)\leq 	\max\left\{\int_{\Omega}z_0,\frac{\beta m_w}{\delta_z}\right\}:= m_z.
\end{equation}	
\end{lemma}

\begin{proof}
 It is easy to see that \eqref{2.3} can be derived through an integration of the first equation in  \eqref{1.3} along with Cauchy--Schwarz's inequality, and
\eqref{2.4} is a direct consequence of $(\alpha_u u+\alpha_w w)v\geq 0$ due to the nonnegativity of $u,w$ and $v$.

In addition, integrating the $w$-equation as well as $u$-equation respectively and  adding the corresponding results, we  then have
\begin{equation}\label{2.7}
\frac{d}{dt}\left(\int_{\Omega}u+\int_{\Omega}w\right)+\int_{\Omega}w+\mu_u\int_{\Omega}u\leq |\Omega|	\mu_u,
\end{equation}
which  readily leads to \eqref{2.5} upon an ODE comparison. At last, thanks to \eqref{2.5}, \eqref{2.6} clearly results from the integration of $z$-equation in  \eqref{1.3}.  	
\end{proof}
Beyond that, making use of the restriction $\beta\in (0,1)$, one can  derive the decay properties of the solution components $w$  and $z$  with respect to $L^1(\Omega)$, which will be used later on.

\begin{lemma}\label{Lemma23}
Suppose that $0<\beta<1$, then there exists constant $C>0$ such that
\begin{equation}\label{2.8}
\int_{\Omega}w(\cdot,t)+\int_{\Omega}z(\cdot,t)\leq C e^{-\delta t}~~\hbox{for all $t>0$.}	
\end{equation}
with $\delta=\min\{1-\beta,\delta_z\}$.	
\end{lemma}
\begin{proof}
We use the  $z$-equation and  $w$-equation to compute
\begin{equation}\label{2.9}
\frac{d}{dt}\left(\int_{\Omega}w+\int_{\Omega}z\right)+(1-\beta)\int_{\Omega}w+\delta_z\int_{\Omega}z=0.
\end{equation}
Due to $0<\beta<1$, this readily implies that
\begin{equation*}
\left(\int_{\Omega}w(\cdot,t)+\int_{\Omega}z(\cdot,t)\right)\leq \left(\int_{\Omega}w_0+\int_{\Omega}z_0\right)e^{-\min\{1-\beta,\delta_z\}t}	
\end{equation*}
and hence  \eqref{2.8} is valid with $C=\int_{\Omega} w_0+\int_{\Omega}z_0$. 	
\end{proof}

\section{Global boundedness}

   As in \cite{LiWang,TW-JDE}, the crucial step in establishing a priori $L^\infty$ bounds for $a,b$ and $z$  is to derive estimates for $a$ and $b$ in $LlogL$, which turn out to be consequences of a quasi-energy structure associated with the system \eqref{2.2} rather than the system \eqref{1.3}. Indeed,  making appropriate use of the logistic degradation in the first equation of \eqref{2.2} and inter alia the $L^1$-decay property of the solution component $w$, one can verify that functional
$$
\mathcal{F}(t):=\int_{\Omega}e^{\chi_u v}a(\cdot,t)\log a(\cdot,t)+\int_{\Omega}e^{\chi_w v}b(\cdot,t)\log b(\cdot,t)+\frac 12\int_{\Omega}z^2(\cdot,t),
$$
which does not involve the Dirichlet integral of $\sqrt{v}$,
actually possesses a certain quasi-dissipative property for all $t>t_0$ with constant $t_0>1$ suitably chosen.
\begin{lemma}\label{Lemma31}
  For any $\varepsilon>0$, there exists $C(\varepsilon)>0$ such that
\begin{equation}\label{3.1}
\begin{array}{rl}
&\displaystyle\frac{d}{dt}\int_{\Omega}e^{\chi_u v}a\log a+\int_{\Omega}e^{\chi_u v}a\log a+D_u\int_{\Omega}e^{\chi_u v}\frac{|\nabla a|^2}{a}
+\frac{\mu_u}{2}\int_{\Omega}e^{2\chi_u v}a^2\log a\\[2mm]
\leq& \displaystyle \varepsilon\int_{\Omega}b^2+C(\varepsilon)
\end{array}
\end{equation}
for all $t>0$.	
\end{lemma}
\begin{proof}
From the first equation in \eqref{1.3}, it follows that
\begin{equation}\label{3.2}
(e^{\chi_u v}a)_t=D_u\nabla \cdot (e^{\chi_u v}\nabla a)+\mu_u e^{\chi_u v}a(1-e^{\chi_u v}a)-\rho e^{\chi_u v}az.	
\end{equation}
Hence, a testing procedure on the first equation in \eqref{2.2} leads to
\begin{equation}\label{3.3}
\begin{array}{rl}
\displaystyle\frac{d}{dt}\int_{\Omega}e^{\chi_u v}a\log a	&
 =\displaystyle\int_{\Omega}(e^{\chi_u v}a)_t\log a+\int_{\Omega}e^{\chi_u v}a_t\nonumber\\
&\leq \displaystyle-D_u\int_{\Omega}e^{\chi_u v}\frac{|\nabla a|^2}{a}+\mu_u\int_{\Omega}e^{\chi_u v}a\log a-\mu_u\int_{\Omega}e^{2\chi_u v}a^2\log a
\nonumber\\
&\displaystyle-\rho\int_{\Omega}e^{\chi_u v}za\log a
+\mu_u\int_{\Omega}ae^{\chi_u v}-\rho\int_{\Omega}e^{\chi_u v}za\nonumber\\
&\displaystyle+\chi_u \alpha_u\int_{\Omega}e^{2\chi_u v}v a^2
+\chi_u \alpha_w\int_{\Omega}e^{(\chi_w+\chi_u)v}vab.
\end{array}
\end{equation}
Thanks to Lemma \ref{Lemma22} and the elementary inequality $a\log a\geq -\frac{1}{e}$ valid for all $a>0$, one can  find $c_1>0$ such that
\begin{equation*}
\begin{array}{rl}
&\displaystyle\frac{d}{dt}\int_{\Omega}e^{\chi_u v}a\log a+D_u\int_{\Omega}e^{\chi_u v}\frac{|\nabla a|^2}{a} +\mu_u\int_{\Omega}e^{2\chi_u v}a^2\log a\\[2mm]
\leq& \displaystyle\alpha_u\chi_ue^{2\chi_u m_v}m_v\int_{\Omega} a^2+\alpha_w\chi_ue^{(\chi_u+\chi_w)m_v}m_v\int_{\Omega}ab +\mu_u\int_{\Omega}e^{\chi_u v}a\log a +c_1,
\end{array}
\end{equation*}
which along with Young's inequality implies that for  any $\varepsilon>0$
\begin{equation*}
\begin{array}{rl}
&\displaystyle\frac{d}{dt}\int_{\Omega}e^{\chi_u v}a\log a+D_u\int_{\Omega}e^{\chi_u v}\frac{|\nabla a|^2}{a} +\mu_u\int_{\Omega}e^{2\chi_u v}a^2\log a\\[2mm]
\leq& (\displaystyle\alpha_u\chi_ue^{2\chi_u m_v}m_v+
\frac1{\varepsilon}
\alpha_w^2\chi_u^2e^{2(\chi_u+\chi_w)m_v}m_v^2 )
\int_{\Omega} a^2+
\varepsilon \int_{\Omega}b^2+ \mu_u\int_{\Omega}e^{\chi_u v}a\log a+ c_1.
\end{array}
\end{equation*}
Further invoking  the inequality $a^2\leq \varepsilon_1 a^2\log a+e^{\frac{2}{\varepsilon_1}}$ for any $\varepsilon_1>0$, we arrive at
\begin{equation}\label{3.3}
\begin{array}{rl}
&\displaystyle\frac{d}{dt}\int_{\Omega}e^{\chi_u v}a\log a+D_u\int_{\Omega}e^{\chi_u v}\frac{|\nabla a|^2}{a}
+\frac{3\mu_u}{4}\int_{\Omega}e^{2\chi_u v}a^2\log a\\
\leq &
\varepsilon \displaystyle\int_{\Omega}b^2+\mu_u\int_{\Omega}e^{\chi_u v}a\log a+ c_2(\varepsilon)	
\end{array}
\end{equation}
with some $c_2(\varepsilon)>0$.
Accordingly, \eqref{3.1} is a consequence of \eqref{3.3} and the fact that $ a\log a\leq \varepsilon_2 a^2\log a- \varepsilon_2^{-1} \ln \varepsilon_2 $
with $\varepsilon_2=\frac{\mu_u}{4(\mu_u+1)}$.
\end{proof}

For the solution component $ b $ of \eqref{2.2}, we also have
\begin{lemma}\label{Lemma32}
Let $0<\beta<1$. Then one can find  $C>0$ and  $t_0>0$  such that for all $t>t_0$
\begin{align}\label{3.4}
&\displaystyle\frac{d}{dt}\int_\Omega e^{\chi_w v}b\log b+ \int_\Omega e^{\chi_w v}b\log b+
\frac {D_w}2 \displaystyle\int_\Omega e^{\chi_w v}\displaystyle\frac{|\nabla b|^2}{b}\nonumber\\
\leq & \int_{\Omega}a^2 +  C(\int_{\Omega} z^4)^{\frac 12}.
  \end{align}
\end{lemma}

\begin{proof}
\rm From the second equation in  \eqref{1.3}, it follows that
\begin{equation*}
(be ^{\chi_w v})_t=D_w\nabla\cdot(e^{\chi_w v}\nabla b)-  w+\rho uz.
\end{equation*}
Relying on $0\leq v\leq m_v$  in $\overline{\Omega}\times (0,\infty)$, a straightforward calculation along with the Young inequality yields
\begin{align}\label{3.5}
&\displaystyle\frac{d}{dt}\int_\Omega e^{\chi_w v}b\log b+ \int_\Omega e^{\chi_w v}b\log b+
D_w\displaystyle\int_\Omega e^{\chi_w v}\displaystyle\frac{|\nabla b|^2}{b}
\nonumber \\
=&\int_\Omega \log (eb) (\rho uz- w) +\chi_w \displaystyle \int_\Omega w(\alpha_u u
 + \alpha_w w )v+ \int_\Omega w\log b
\\[1mm]
 \leq & \alpha_w\chi_w e^{2\chi_w m_v}m_v\int_{\Omega}b^2+ \rho \int_{\Omega}aze^{\chi_u v}\log  b
 + \rho e^{\chi_u m_v}\int_{\Omega}az
 +\alpha_u\chi_w e^{(\chi_u+\chi_w)m_v}m_v\int_{\Omega}ab
  \nonumber\\
  \leq & c_1\int_{\Omega}b^2+\int_{\Omega}a^2 +  c_1\int_{\Omega}z^2+ c_1\displaystyle\int_{\{x\in \Omega;b(x,t)\geq 1\}}z^2 \log ^2 b
  \nonumber\\
  \leq & c_1\int_{\Omega}b^2+\int_{\Omega}a^2 +  c_1\int_{\Omega}z^2+ c_1\displaystyle(\int_{\Omega} z^4)^{\frac 12}
  (\int_{\{x\in \Omega;b(x,t)\geq 1\}}|\log  b|^4) ^{\frac 12} \nonumber\\
 \leq & c_1\int_{\Omega}b^2+\int_{\Omega}a^2 +  c_1\int_{\Omega}z^2+ c_1\displaystyle(\int_{\Omega} z^4)^{\frac 12}
  (\int_{\Omega} b+c_2|\Omega|)^{\frac 12}\nonumber
  \end{align}
with some $c_1>0$, where we use the fact that  there exits $c_2>0$ such that $\log^{4} s\leq s+c_2 $ for all $s\geq 1$.

Furthermore, in order to appropriately   estimate the first summand  on the right-hand of \eqref{3.5}, we apply the two-dimensional Gagliardo--Nirenberg inequalities
\begin{equation*}
\|\varphi\|^4_{L^4(\Omega)}
\leq C_{g}\|\nabla \varphi\|_{L^2(\Omega)}^2\|\varphi\|^2_{L^2(\Omega)}+C_g\|\varphi\|^4_{L^2(\Omega)} ~~~~\hbox{for some}~ C_g>0 ~\hbox{and all}~ \varphi\in W^{1,2}(\Omega)
\end{equation*}
  to get
  \begin{equation}
\begin{array}{rl}\label{3.6}
 c_1\displaystyle \int_{\Omega}b^2=
 & c_1\|\sqrt{b}\|^4_{L^4(\Omega)}\\
 \leq &
 \displaystyle c_1C_g \displaystyle\int_\Omega b
 \displaystyle\int_\Omega e^{\chi_w v}\displaystyle\frac{|\nabla b|^2}{b}+ c_1C_g (\displaystyle\int_\Omega b)^2\\
 \leq &
 \displaystyle c_1C_g \displaystyle\int_\Omega w
 \displaystyle\int_\Omega e^{\chi_w v}\displaystyle\frac{|\nabla b|^2}{b}+ c_1C_g (\displaystyle\int_\Omega w)^2\\
   \end{array}
\end{equation}
Therefore combining  \eqref{3.6} with \eqref{3.5}, we arrive at
\begin{align}\label{3.7}
&\displaystyle\frac{d}{dt}\int_\Omega e^{\chi_w v}b\log b+ \int_\Omega e^{\chi_w v}b\log b+
D_w\displaystyle\int_\Omega e^{\chi_w v}\displaystyle\frac{|\nabla b|^2}{b}
 \\
\leq & \displaystyle c_1C_g \displaystyle\int_\Omega w
 \displaystyle\int_\Omega e^{\chi_w v}\displaystyle\frac{|\nabla b|^2}{b}+ c_1C_g (\displaystyle\int_\Omega w)^2+\int_{\Omega}a^2 +  c_1\int_{\Omega}z^2+ c_1\displaystyle(\int_{\Omega} z^4)^{\frac 12}
  (\int_{\Omega} w+c_2|\Omega|)^{\frac 12}.\nonumber
  \end{align}
By Lemma \ref{Lemma23}, we can pick $t_0>0$ suitably large such that
$$
 c_1C_g \displaystyle\int_\Omega w\leq \frac {D_w}2
$$
and thereby  for $t\geq t_0$,
\begin{align*}\label{3.7}
&\displaystyle\frac{d}{dt}\int_\Omega e^{\chi_w v}b\log b+ \int_\Omega e^{\chi_w v}b\log b+
\frac {D_w}2 \displaystyle\int_\Omega e^{\chi_w v}\displaystyle\frac{|\nabla b|^2}{b}
\nonumber \\
\leq &  c_1C_g (\displaystyle\int_\Omega w)^2+\int_{\Omega}a^2 +  c_1\int_{\Omega}z^2+ c_1\displaystyle(\int_{\Omega} z^4)^{\frac 12}
  (\int_{\Omega} w+c_2|\Omega|)^{\frac 12},
  \end{align*}
which along with Lemma \ref{Lemma22} and the Young inequality completes the proof. \end{proof}

Whereas the expressions  $\int_{\Omega}a^2$ appearing in  \eqref{3.4}   turns out to be conveniently digestible through the dissipation rate
 in \eqref{3.1}, it remains to  estimate  $(\int_{\Omega} z^4)^{\frac 12}$ by means of an interpolation argument.
\begin{lemma}\label{Lemma33}
Let $0<\beta<1$ and define
$$
\mathcal{F}(t):=\int_{\Omega}e^{\chi_u v}a(\cdot,t)\log a(\cdot,t)+\int_{\Omega}e^{\chi_w v}b(\cdot,t)\log b(\cdot,t)+\int_{\Omega}z^2(\cdot,t).
$$
Then  there exists  constant $C>0$ such that
\begin{equation}\label{3.8}
\mathcal{F'}(t)	+\mathcal{F}(t)\leq   C
\end{equation}
for all $t\geq t_0$ as given in Lemma \ref{Lemma32}.

\end{lemma}
\begin{proof}
Testing the fourth equation in \eqref{2.2} by $z$, we get
\begin{equation*}
\begin{array}{rl}
\displaystyle\frac{d}{dt}\int_{\Omega}z^2+\int_{\Omega}z^2+2D_z\int_{\Omega}|\nabla z|^2&\leq
\displaystyle 2\beta\int_{\Omega}e^{\chi_w v}bz +\int_{\Omega}z^2\\[2mm]
&\leq\displaystyle  \varepsilon\int_{\Omega}b^2+(1+\frac {e^{2\chi_w m_v}\beta^2}{\varepsilon})\int_{\Omega}z^2.	
\end{array}
\end{equation*}
Hence  Lemma \ref{Lemma31} and Lemma \ref{Lemma32} provide positive constants $c_i(\varepsilon)>0 (i=1,2)$ such that
\begin{equation}\label{3.9}
\begin{array}{rl}
&\mathcal{F'}(t)	+\mathcal{F}(t)
+\displaystyle D_w\int_{\Omega}e^{\chi_w v}\frac{|\nabla b|^2}{b}+D_z\int_{\Omega}|\nabla z|^2+\frac{\mu_u}{2}\int_{\Omega}e^{2\chi_u v}a^2\log a+
\int_{\Omega}z^2\\
\leq & \displaystyle\int_{\Omega}a^2 +  c_1(\varepsilon)(\int_{\Omega} z^4)^{\frac 12}+  2\varepsilon\int_{\Omega}b^2+c_2(\varepsilon)
\end{array}
\end{equation}
for all $t\geq t_0$ with $t_0$ given by Lemma \ref{Lemma32}.

Now again since $a^2\leq \varepsilon_1 a^2\log a+e^{\frac{2}{\varepsilon_1}}$ for any $\varepsilon_1>0$,
\begin{equation}\label{3.10}
\displaystyle\int_{\Omega}a^2 \leq \frac{\mu_u}{4}\int_{\Omega}e^{2\chi_u v}a^2\log a+c_3
\end{equation}
with $c_3>0$,
whereas according to the two-dimensional Gagliardo--Nirenberg inequalities,
 \begin{equation}
\begin{array}{rl}\label{3.11}
  2\varepsilon\displaystyle \int_{\Omega}b^2=
 &  2\varepsilon\|\sqrt{b}\|^4_{L^4(\Omega)}\\
 \leq &
 \displaystyle  2\varepsilon C_g \displaystyle\int_\Omega b
 \displaystyle\int_\Omega e^{\chi_w v}\displaystyle\frac{|\nabla b|^2}{b}+  2\varepsilon C_g (\displaystyle\int_\Omega b)^2\\
 \leq &
 \displaystyle   2\varepsilon m_w
 \displaystyle\int_\Omega e^{\chi_w v}\displaystyle\frac{|\nabla b|^2}{b}+ 2\varepsilon m_w^2 C_g \\
 \leq &
 \displaystyle \frac{D_w}2\int_{\Omega}e^{\chi_w v}\frac{|\nabla b|^2}{b}+c_4
    \end{array}
\end{equation}
by the choice of $\varepsilon=\frac{D_w}{4m_w} $.

In summary, \eqref{3.11}, \eqref{3.10} and \eqref{3.9} show that for $t\geq t_0$
\begin{equation}\label{3.12}
\mathcal{F'}(t)	+\mathcal{F}(t)
+D_z\int_{\Omega}|\nabla z|^2+\int_{\Omega}z^2\leq   c_1(\int_{\Omega} z^4)^{\frac 12}+  c_5
\end{equation}
with some $c_5>0$.

 Furthermore to estimate $\left(\int_{\Omega}z^4\right)^{\frac{1}{2}}$ on the right-hand of  \eqref{3.12}, we employ the Gagliardo--Nirenberg inequalities once more to obtain that
\begin{equation}\label{3.13}
\begin{array}{rl}
c_1\|z(\cdot,t)\|^2_{L^4(\Omega)}&
\leq c_1 c_g
\|\nabla z(\cdot,t)\|^{\frac 32}_{L^2(\Omega)}\|z(\cdot,t)\|^{\frac 12}_{L^1(\Omega)} +c_1 c_g\|z(\cdot,t)\|^2_{L^1(\Omega)}\\
&\leq \displaystyle \frac{D_z}2\int_{\Omega}|\nabla z|^2+c_6
\end{array}
\end{equation}
with $c_6>0$,~which together with  \eqref{3.12} readily establishes  \eqref{3.8}.
\end{proof}
As a consequence of \eqref{3.8},  the $LlogL$-estimate of quantities $a$
and $b$  is  achieved  as follows
\begin{lemma}\label{Lemma34}
Let $0<\beta<1$. Then there exists $C> 0$ such that for all $t> t_0$,
\begin{equation}\label{3.14}
\int_{\Omega}a(\cdot,t)|\log a(\cdot,t)|\leq C	
\end{equation}
as well as
\begin{equation}\label{3.15}
\int_{\Omega}b(\cdot,t)|\log b(\cdot,t)|\leq C.	
\end{equation}	
\end{lemma}
\begin{proof}  As in Lemma 3.4 of \cite{LiWang}, by the  inequality $a\log a>-e^{-1}$ in $\Omega\times (0,\infty)$, we have
\begin{equation*}
\begin{array}{rl}
\displaystyle\int_{\Omega}a(\cdot,t)|\log a(\cdot,t)|&\leq \displaystyle\int_{\Omega}e^{\chi_u v}a(\cdot,t)\log a(\cdot,t)-2\int_{a<1}e^{\chi_{u} v}a(\cdot,t)\log a(\cdot,t)	\\[2mm]
&\leq \displaystyle\int_{\Omega}e^{\chi_u v}a(\cdot,t)\log a(\cdot,t)+\frac{2|\Omega|e^{\chi_u m_v}}{e}.
\end{array}
\end{equation*}
Likewise,~we can also obtain
\begin{equation*}
\int_{\Omega}b(\cdot,t)|\log b(\cdot,t)|\leq \int_{\Omega}e^{\chi_w v}b(\cdot,t)\log b(\cdot,t)+\frac{2|\Omega|e^{\chi_w m_v}}{e}.	
\end{equation*}
According to \eqref{3.8},  there exits $c_1>0$ such that
\begin{equation}\label{3.16}
\mathcal{F}(t)\leq c_1.
\end{equation}
Therefore by the definition of   $\mathcal{F}(t)$,
we can see that
\begin{eqnarray*}
\int_{\Omega}a(\cdot,t)|\log a(\cdot,t)|+\int_{\Omega}b(\cdot,t)|\log b(\cdot,t)|\leq \mathcal{F}(t)+
\frac{2|\Omega|e^{\chi_u m_v}}{e}+\frac{2|\Omega|e^{\chi_w m_v}}{e}
\leq c_2
\end{eqnarray*}
with $c_2=
\frac{2|\Omega|}
e(e^{\chi_w m_v}
+e^{\chi_u m_v} )+ c_1$ for all $t>t_0$, and thus complete the proof.
\end{proof}

The a priori estimates for $a, b$  gained in Lemma \ref{Lemma34} is the cornerstone to
 establish a $L^{\infty}(\Omega)$-bound for solution $(u,v,w,z)$.  Indeed,  one can proceed to derive obtain $L^{\infty}(\Omega)$-bound
 by means of some quite straightforward  $L^p$ testing procedures.
\begin{lemma}\label{Lemma35}
Let $0<\beta<1$. Then there exists $C> 0$ such that  for all $t>0$
\begin{equation}\label{3.17}
\|a(\cdot,t)\|_{L^{\infty}(\Omega)}+\|b(\cdot,t)\|_{L^{\infty}(\Omega)}+\|z(\cdot,t)\|_{L^{\infty}(\Omega)}
\leq C.
\end{equation}
\end{lemma}

\begin{proof}
 Testing the first equation in \eqref{2.2} by $e^{\chi_u v}a^{r-1}$ with $r\geq 2$, we obtain $c_1(r)>0$
\begin{equation*}
\begin{array}{rl}
&\displaystyle\frac{d}{dt}\int_{\Omega}e^{\chi_u v}a^r+\frac{4(r-1)D_u}{r}\int_{\Omega}e^{\chi_u v}|\nabla a^{\frac{r}{2}}|^2+
r\mu_u\int_{\Omega}e^{2\chi_v}a^{r+1}\\
\leq &\displaystyle\mu_u r\int_{\Omega}e^{\chi_u v}a^r+r \chi_u \int_{\Omega}e^{\chi_u v}a^r(\alpha_u u +\alpha_w w)v\\
\leq &\displaystyle\mu_u r\int_{\Omega}e^{\chi_u v}a^r+c_1(r) \int_{\Omega}a^{r}(a +b),
\end{array}
\end{equation*}
which together with the Young inequality, leads to
\begin{equation}\label{3.18}
\begin{array}{rl}
&\displaystyle\frac{d}{dt}\int_{\Omega}e^{\chi_u v}a^r+2 D_u \int_{\Omega}|\nabla a^{\frac{r}{2}}|^2+
\int_{\Omega}e^{\chi_u v}a^r\\
\leq & \displaystyle c_2(r)\int_{\Omega}a^{r+1}+c_2(r)\int_{\Omega}b^{r+1}+c_2(r)
\end{array}
\end{equation}
with constant $c_2(r)>0$.
Likewise, there exist  $c_i(r)>0 (i=3,4)$ such that
\begin{equation}\label{3.19}
\begin{array}{rl}
&\displaystyle\frac{d}{dt}\int_{\Omega}e^{\chi_w v}b^r+2 D_w \int_{\Omega}|\nabla b^{\frac{r}{2}}|^2+
r\int_{\Omega}e^{\chi_v}b^r\\
\leq & \displaystyle c_3(r)\int_{\Omega}b^{r-1}az +c_3(r)\int_{\Omega}a^{r+1} +c_3(r)\int_{\Omega}b^{r+1}\\[2mm]
\leq & \displaystyle c_4(r)\int_{\Omega} a^{r+1}+c_4(r)\int_{\Omega}b^{r+1} +c_4(r)\int_{\Omega}z^{r+1}
\end{array}
\end{equation}
as well as
\begin{eqnarray}\label{3.20}
\frac{d}{dt}\int_{\Omega}z^{r}+2 D_z \int_{\Omega}|\nabla z^{\frac{r}{2}}|^2+\frac{r\delta_z}{2}\int_{\Omega}z^{r}\leq \frac{2r\beta^{r}}{\delta_z}\int_{\Omega}w^{r}.	
\end{eqnarray}
Collecting \eqref{3.18}--\eqref{3.20},  we then arrive at
\begin{equation}\label{3.21}
\begin{array}{rl}
& \displaystyle \frac{d}{dt}
\int_{\Omega}(e^{\chi_u v}a^r+e^{\chi_w v} b^r+ z^{r})
+c_5
\int_{\Omega}(e^{\chi_u v}a^r+e^{\chi_w v} b^r+ z^{r})
+c_5 \int_{\Omega}(|\nabla a^{\frac{r}{2}}|^2+|\nabla b^{\frac{r}{2}}|^2
+|\nabla z^{\frac{r}{2}}|^2) \\
\leq &
\displaystyle c_6(r)(\int_{\Omega}a^{r+1}+\int_{\Omega}b^{r+1}+\int_{\Omega}z^{r+1})+c_6(r)
\end{array}
\end{equation}
with $c_5>0$ and $c_6(r)>0$.

Now we invoke the logarithm-type Gagliardo--Nirenberg inequality (we refer to Lemma A.5 in \cite{TW-JDE} for details) to obtain that for
any $\varepsilon>0$,
\begin{equation}\label{3.22}
\int_{\Omega}a^{r+1}\leq \varepsilon\|\nabla a^{\frac{r}{2}}\|^2_{L^2(\Omega)}\cdot\int_{\Omega}a\cdot|\log |a^{\frac{r}{2}}||+C(\varepsilon,r)\left(\|a^{\frac{r}{2}}\|^{\frac{2(r+1)}{r}}_{L^{\frac{2}{r}}(\Omega)}+1\right),
\end{equation}
and
\begin{equation}\label{3.23}
\int_{\Omega}b^{r+1}\leq \varepsilon\|\nabla b^{\frac{r}{2}}\|^2_{L^2(\Omega)}\cdot\int_{\Omega}b\cdot|\log |b^{\frac{r}{2}}||+C(\varepsilon,r)\left(\|b^{\frac{r}{2}}\|^{\frac{2(r+1)}{r}}_{L^{\frac{2}{r}}(\Omega)}+1\right)	
\end{equation}
as well as
\begin{equation}\label{3.24}
\int_{\Omega}z^{r+1}\leq \varepsilon\|\nabla z^{\frac{r}{2}}\|^2_{L^2(\Omega)}\cdot\int_{\Omega}z\cdot|\log |z^{\frac{r}{2}}||+C(\varepsilon,r)\left(\|z^{\frac{r}{2}}\|^{\frac{2(r+1)}{r}}_{L^{\frac{2}{r}}(\Omega)}+1\right).
\end{equation}
Therefore combining \eqref{3.21}--\eqref{3.24} with  Lemma \ref{Lemma34}, one can find  $c_7(r)>0$ and $c_8(\varepsilon,r)>0$ such that
\begin{equation}\label{3.25}
\begin{array}{rl}
& \displaystyle \frac{d}{dt}
\int_{\Omega}(e^{\chi_u v}a^r+e^{\chi_w v} b^r+ z^{r})
+c_5
\int_{\Omega}(e^{\chi_u v}a^r+e^{\chi_w v} b^r+ z^{r})
+c_5 \int_{\Omega}(|\nabla a^{\frac{r}{2}}|^2+|\nabla b^{\frac{r}{2}}|^2
+|\nabla z^{\frac{r}{2}}|^2) \\[2mm]
\leq &
\displaystyle  c_7(r) \varepsilon \int_{\Omega}(|\nabla a^{\frac{r}{2}}|^2+|\nabla b^{\frac{r}{2}}|^2
+|\nabla z^{\frac{r}{2}}|^2)
+c_8(\varepsilon, r).
\end{array}
\end{equation}
Accordingly  upon the choice of $\varepsilon=\frac{c_5}{2c_7(r)}$,  \eqref{3.25} shows that
\begin{eqnarray*}
\displaystyle \frac{d}{dt}
\int_{\Omega}(e^{\chi_u v}a^r+e^{\chi_w v} b^r+ z^{r})
+c_5
\int_{\Omega}(e^{\chi_u v}a^r+e^{\chi_w v} b^r+ z^{r})\leq c_9(r)
\end{eqnarray*}
with some $c_9(r)>0$ for all $t>t_0$, and hence entails
\begin{equation*}
\int_{\Omega}a^r(\cdot,t)+\int_{\Omega}b^r(\cdot,t)+\int_{\Omega}z^r(\cdot,t)\leq c_{10}(r)	
\end{equation*}
with some $c_{10}(r)> 0$ by a standard ODE comparison argument.
At this position, one
can  derive a bound for $a,b,z$ with respect to the norm in $L^\infty(\Omega)$ by means of a Moser-type iteration argument in quite a standard manner. We omit the proof thereof, and would like refer to \cite{LiWang,TW-NA,TW-JDE} for details in a closely related setting.
\end{proof}

\section{Asymptotic behavior}
On the basis of the exponential decay of  quantities $w,z$ with respect to the norm in $L^1(\Omega)$ and global boundedness of solutions, we will address  the large time asymptotics of the solution $(u,v,w,z)$ to \eqref{1.3}.
To this end, we first turn the $L^1$-decay information explicitly contained in Lemma \ref{Lemma23} to the decay property of $z$ in $L^\infty$-norm
by an appropriate application of the parabolic smoothing estimates in the two-dimensional domain.

\begin{lemma}\label{Lemma41}
Let $(u,v,w,z)$ be the global solutions of \eqref{1.3} with $\beta\in (0,1)$. Then for all $t>0$, we have
\begin{equation}\label{4.1}
\|z(\cdot,t)\|_{L^{\infty}(\Omega)}\leq Ce^{- \gamma_1 t}
\end{equation}
with $\gamma_1=\frac {\min\{1-\beta,\delta_z\}}2$ and some $C>0$.
\end{lemma}

\begin{proof}
We invoke Lemma \ref{Lemma23} along with  \eqref{3.17}  to see that there exist $c_1>0$ and $c_2>0$ such that
\begin{equation}\label{4.2}
\|w(\cdot,t)\|_{L^2(\Omega)}^2\leq c_1\|w(\cdot,t)\|_{L^\infty(\Omega)} e^{-\delta t}\leq c_2 e^{-\delta t}
\end{equation}
with $\delta=\min\{1-\beta,\delta_z\}>0$.
According to known smoothing properties of the the Neumann heat semigroup  $\left ( e^{\sigma\Delta } \right )_{\sigma> 0}$ on the domain $  \Omega\subset \mathbb{R}^2$ (\cite{WinklerJDE}), there exists $c_3 >0$ such that  for each $\varphi\in  C^0(\Omega)$,
\begin{equation}\label{4.3}
\left \| e^{\sigma D_z\Delta } \varphi\right \|_{L^{\infty}(\Omega)}
\leq c_3( 1+\sigma^{-\frac 12})
 \| \varphi  \|_{L^2(\Omega)}.
\end{equation}
Due to  the nonnegativity of $z$ and the comparison principle, we may use \eqref{4.2} and
\eqref{4.3} to infer that
\begin{equation}\label{4.4}
\begin{array}{rl}
z(\cdot,t)&=
e^{t(D_z\Delta-\delta_z)}z_0+\displaystyle\int^{t}_0e^{(t-s)(D_z\Delta-\delta_z)}(\beta w-\rho uz)(\cdot,s)ds	\\
&\leq
e^{t(D_z\Delta-\delta_z)}z_0+\beta \displaystyle\int^{t}_0e^{(t-s)(D_z\Delta-\delta_z)}w(\cdot,s)ds	\\
& \leq \displaystyle e^{-\delta_z t}\|z_0\|_{L^{\infty}(\Omega)}+\beta c_3
\int^{t}_0(1+(t-s)^{-\frac12})e^{-\delta_z(t-s)}\|w(\cdot,s)\|_{L^{2}(\Omega)}ds\\
&\leq  \displaystyle e^{-\delta_z t}\|z_0\|_{L^{\infty}(\Omega)}+\beta c_2 c_3
\int^{t}_0(1+(t-s)^{-\frac12})e^{-\delta_z(t-s)} e^{-\frac{\delta s}2}ds\\
&\leq  \displaystyle e^{-\delta_z t}\|z_0\|_{L^{\infty}(\Omega)}+\beta c_2 c_3 c_4
e^{-min\{\delta_z, \frac{\delta }2)t }
\end{array}
\end{equation}
with  some $c_4>0$, which along with the  nonnegativity of $z$  entails  that \eqref{4.1} holds with $C=\|z_0\|_{L^{\infty}(\Omega)}
+\beta c_2 c_3 c_4.$
\end{proof}

Now thanks to the uniform decay property of $z$, a pointwise lower bound for $a=ue^{-\chi_u v}$ can be achieved
by means of an argument based on comparison with spatially flat functions,  which is documented as
follows.
\begin{lemma}\label{Lemma42}
Under the assumption of Lemma \ref{Lemma41},~one can find  constant $\gamma>0$ such that
\begin{equation}\label{4.5}
a(x,t)>\gamma~~\hbox{for all $(x,t)\in\Omega\times(0,\infty)$}.
\end{equation}	
\end{lemma}
\begin{proof}
According to Lemma \ref{Lemma41}, one can pick $t_0>0$ sufficiently large  such that for all $t>t_0$
\begin{equation}\label{4.6}
\|z(\cdot,t)\|_{L^{\infty}(\Omega)}\leq \frac{\mu_u}{2\rho}.
\end{equation}
Hence by means of a straightforward computation based
on \eqref{2.2},   one can see that
\begin{eqnarray*}
a_t&\geq &D_ue^{-\chi_u v}\nabla\cdot(e^{\chi_u v}\nabla a)+\mu_u a(1-ae^{\chi_u v})-\displaystyle\rho az\\
&\geq &D_ue^{-\chi_u v}\nabla \cdot(e^{\chi_u v}\nabla a)+a\left(\frac{\mu_u}{2}-e^{\chi_u \|v_0\|_{L^{\infty}(\Omega)}}a \right)
\end{eqnarray*}
for all $t>t_0$.

Now let $\underline{a}(t)$ be the smooth solution to the initial value problem:
\begin{equation}\label{4.7}
\left\{
\begin{array}{lll}
\underline{a}_t=\underline{a}\left(\frac{\mu_u}{2}-e^{\chi_u \|v_0\|_{L^{\infty}(\Omega)}}\underline{a}\right),\\
\underline{a}(t_0)=\inf\limits_{x\in{\Omega}}
\{ u(x,t_0)	e^{-\chi_u \|v_0\|_{L^{\infty}(\Omega)}}\},
\end{array}
\right.	
\end{equation}
then  through the explicit solution of above Bernoulli-type ODE, we have
\begin{equation}\label{4.8}
\underline{a}(t)\geq c_1\coloneqq \min\{\underline{a}(t_0),\frac{\mu_u}{2}e^{-\chi_u \|v_0\|_{L^{\infty}(\Omega)}}\}>0	
\end{equation}
for all $t>t_0$.
It is observed that
\begin{eqnarray*}
\underline{a}_t=D_u e^{-\chi_u v}\nabla \cdot(e^{\chi_u v}\nabla \underline{a})+	\underline{a}\left(\frac{\mu_u}{2}-e^{\chi_u \|v_0\|_{L^{\infty}(\Omega)}}\underline{a}\right)
\end{eqnarray*}
and $a(x,t_0)\geq \underline{a}(t_0)$. Hence from the comparison principle of the parabolic equation, one can conclude that
\begin{equation}\label{4.9}
a(x,t)\geq \underline{a}	(t)\geq c_1\qquad \hbox{for all $(x,t)\in \Omega\times(t_0,\infty)$}.
\end{equation}
On the other hand,~for any $t\in[0,t_0]$,~the continuity of $a$ allows us to find  constant $c_2\coloneqq \inf\limits_{(x,t)\in\overline{\Omega}\times[0,t_0]}a(x,t)$ such that
\begin{equation}\label{4.10}
a(x,t)\geq c_2\qquad\hbox{for all $(x,t)\in\Omega\times(0,t_0]$}.	
\end{equation}
and thereby \eqref{4.5} results from \eqref{4.9} and \eqref{4.10} with $\gamma \coloneqq \min\{c_1,c_2\}$.
\end{proof}

In view of the $v$-equation in \eqref{2.2}, the latter information immediately entails  the
 exponential decay of $v$  with respect to  $L^{\infty}(\Omega)$ norm.

\begin{lemma}\label{Lemma43}
Under the assumption of Lemma \ref{Lemma41}, we have for all $t>0$
\begin{equation}\label{4.11}
\|v(\cdot,t)\|_{L^{\infty}(\Omega)}\leq Ce^{-\gamma_2 t}	
\end{equation}
with $\gamma_2=\alpha_u\gamma$ and some $C>0$.
\end{lemma}
\begin{proof}
By recalling the outcomes of Lemma \ref{Lemma42},~we have
\begin{equation*}
v_t=-(\alpha_u u+\alpha_w w)v\leq -\alpha_u\gamma v	
\end{equation*}
and hence
$$
v(x,t)\leq v_0(x) e^{-\alpha_u\gamma t}	
$$
from which \eqref{4.11} follows.	\end{proof}

Furthermore upon the  decay property of $v$ with respect to $L^\infty(\Omega)$,  one can derive the
following basic stabilization feature of $a(=e^{-\chi_u v}u)$.
\begin{lemma}\label{Lemma44}
Let the assumption of Lemma \ref{Lemma41} hold. Then we have
\begin{equation}\label{4.12}
\int^{\infty}_0\int_{\Omega}\frac{|\nabla a|^2}{a^2}< \infty
	\end{equation}
as well as
\begin{equation}\label{4.13}
\int^{\infty}_0\int_{\Omega}(u-1)^2< \infty.
	\end{equation}
\end{lemma}
\begin{proof}
In view of $s-1-\log s>0$ for all $s>0$ and $v_t<0$, we can conclude that
\begin{equation}\label{4.14}
\begin{array}{rl}
\displaystyle\frac{d}{dt}\int_{\Omega}e^{\chi_u v}(a-1-\log a)=
& \displaystyle\int_{\Omega}e^{\chi_u v}(a-1-\log a)v_t+\int_{\Omega}e^{\chi_u v}\left(\frac{a-1}{a}\right)a_t\\[2mm]
\leq& \displaystyle-D_u\int_{\Omega}e^{\chi_u v}\frac{|\nabla a|^2}{a^2}+\mu_u\int_{\Omega}	e^{\chi_u v}(a-1)(1-u)\\
& +\displaystyle\chi_u\int_{\Omega}e^{\chi_u v}(a-1)(\alpha_uu+\alpha_w w)v-
\displaystyle\rho\int_{\Omega}e^{\chi_u v}z(a-1).
\end{array}
\end{equation}
Here by Young's inequality,
\begin{equation}\label{4.15}
\begin{array}{rl}
(1-a)(1-u)&=(1-u)^2+(u-a)(1-u)\\
&\geq \displaystyle\frac 12 (1-u)^2- a^2(e^{\chi_u v}-1)^2.
\end{array}
\end{equation}
Due to the fact that $e^s\leq 1+2s$ for all $s\in [0,log2]$,  \eqref{4.11} allows us to fix a $t_1>1$ suitable large such that for all $t\geq t_1$,
$$
(e^{\chi_u v(\cdot,t)}-1)^2\leq 4\chi_u ^2 v^2(\cdot,t),
$$
which together with \eqref{4.15} entails that for $t\geq t_1$
$$
(1-a)(1-u)\geq \displaystyle\frac 12 (1-u)^2-4 a^2\chi_u ^2 v^2.
$$
Therefore we infer from \eqref{3.17} and \eqref{4.14} that for all $t\geq t_1$
\begin{equation*}
\begin{array}{rl}
&\displaystyle \frac{d}{dt}\int_{\Omega}e^{\chi_u v}(a-1-\log a)+D_u\int_{\Omega}e^{\chi_u v}\frac{|\nabla a|^2}{a^2}+
\mu_u\int_{\Omega}(u-1)	 ^2\\[2mm]
\leq &\displaystyle
c_1\int_{\Omega}z+c_1\int_{\Omega}v
\end{array}
\end{equation*}
with some $c_1>0$. After a time integration this leads to
\begin{equation*}
\begin{array}{rl}
&\displaystyle D_u\int^t_{t_1}\int_{\Omega}e^{\chi_u v}\frac{|\nabla a|^2}{a^2}+
\mu_u\int^{t}_{t_1} \int_{\Omega}(u-1)^2\\[2mm]
\leq & c_1 \displaystyle\int^{t}_{t_1}\int_{\Omega}z+c_1\int^t_{t_1}\int_{\Omega}v+
\int_{\Omega}e^{\chi_u m_v}(a(\cdot,t_1)-1-\log a (\cdot,t_1))
\end{array}
\end{equation*}
and thereby  implies that  both \eqref{4.12} and \eqref{4.13} is valid  thanks to \eqref{4.1} and \eqref{4.11}.	
\end{proof}

In order to improve yet quite weak decay information of $u$, we turn to consider the exponential decay
 properties of  $\int_{\Omega}|\nabla v(\cdot,t)|^2$, rather than the integrability  of   $a_{t}$ in  $L^2((0,\infty); L^2(\Omega))$.  As the first step toward this, we first show the convergence  of integral $\int^{\infty}_0\int_{\Omega}|\nabla v|^2$, which is stated below.
\begin{lemma}\label{Lemma45}
Let the assumption of Lemma \ref{Lemma41} hold. Then we have
\begin{equation}\label{4.16}
\int^{\infty}_0\int_{\Omega}|\nabla v|^2<\infty.	
\end{equation}	
\end{lemma}

\begin{proof}
Multiplying the second equation in \eqref{2.2} by $e^{\chi_w v}b$ and integrating by parts, one
can conclude that
\begin{eqnarray*}
\frac{d}{dt}\int_{\Omega}e^{\chi_w v}b^2+2D_w\int_{\Omega}e^{\chi_w v}|\nabla b|^2+2 \int_{\Omega}e^{\chi_w v}b^2\\
=2\rho\int_{\Omega}abz e^{\chi_w v}+2\chi_w\int_{\Omega}e^{\chi_w v}b^2(\alpha_u a e^{\chi_u v}+\alpha_w b e^{\chi_w v})v,	
\end{eqnarray*}
which together with the global-in-time boundedness property of $a$ and $b$, implies that	
\begin{eqnarray*}
\frac{d}{dt}\int_{\Omega}e^{\chi_w v}b^2+2D_w\int_{\Omega}e^{\chi_w v}|\nabla b|^2+2\int_{\Omega}e^{\chi_w v}b^2
\leq c_1\left(\int_{\Omega}v+\int_{\Omega}z\right)	
\end{eqnarray*}
for some $c_1>0$.
Hence according to Lemma \ref{Lemma43} and Lemma \ref{Lemma23}, we can get
\begin{equation}\label{4.17}
\int^{\infty}_0\int_{\Omega}|\nabla b|^2<\infty.	
\end{equation}
Now since
 $$
 \nabla v_t=-(\alpha_u \nabla u + \alpha_w \nabla w )v-(\alpha_u u + \alpha_w  w ) \nabla v,
 $$
a direct computation shows that
\begin{equation*}
\begin{array}{rl}
&\displaystyle \frac{1}{2}\frac{d}{dt}\int_{\Omega}|\nabla v|^2+\int_{\Omega}(u+w)|\nabla v|^2\\
=&-\alpha_w \chi_w\displaystyle \int_{\Omega}v e^{\chi_w v} b|\nabla v|^2-
\alpha_w\displaystyle\int_{\Omega}v e^{\chi_w v}\nabla v\cdot \nabla b-
\alpha_u \displaystyle\int_{\Omega}v\nabla v\cdot \nabla u\\
\leq&-\alpha_w\displaystyle\int_{\Omega}v e^{\chi_w v}\nabla v\cdot \nabla b
-\alpha_u \displaystyle\int_{\Omega}v   e^{\chi_u v}\nabla v\cdot \nabla a.	
\end{array}
\end{equation*}
Therefore, recalling the pointwise lower bound in \eqref{4.5}, \eqref{3.17} and by the Young inequality, we can find
find a constant $c_2>0$ such that
\begin{equation}\label{4.18}
\frac{d}{dt}\int_{\Omega}|\nabla v|^2+\gamma \int_{\Omega}|\nabla v|^2\leq c_2
\left(
\int_{\Omega}|\nabla b|^2+
\int_{\Omega}\frac{|\nabla a|^2}{a^2}
\right).	
\end{equation}
Hence combining this with \eqref{4.17} and  \eqref{4.12}, \eqref{4.16} is  actually valid.
\end{proof}

Beyond the integrability of $\int_{\Omega}|\nabla v|^2$ over $(0,\infty)$, we make use of the explicit expression  of
$\nabla v$ together with \eqref{4.17} and  \eqref{4.12}  to identify that  $\int_{\Omega}|\nabla v|^2$  exponentially decays.

\begin{lemma}\label{Lemma46}
Under the assumption of Lemma \ref{Lemma41},
one can find constant $C>0$
such that
\begin{equation}\label{4.19}
\int_{\Omega}|\nabla v|^2\leq C(t+1)e^{-2\gamma t}~~\hbox{for all $t>0$,}	
\end{equation}
where $\gamma $ is given by Lemma \ref{Lemma42}.	
\end{lemma}
\begin{proof}
On the basis of the $v$-equation in \eqref{1.3}, we have
\begin{equation*}
\nabla v(\cdot,t)= \nabla v(\cdot,0)e^{-\int^{t}_0(u+w)(\cdot,s)ds} -v(\cdot,0)e^{-\int^{t}_0(u+w)(\cdot,s)ds}
\int^{t}_0(\nabla u(\cdot,s)+\nabla w(\cdot,s))ds	
\end{equation*}
which along with \eqref{4.5} and the Young inequality entails that
\begin{equation}\label{4.20}
\int_{\Omega}|\nabla v|^2\leq
2e^{-2\gamma t}\|\nabla v_0\|^2_{L^2(\Omega)}+4 te^{-2\gamma t}\|v_0\|^2_{L^\infty(\Omega)}
(\int^{t}_0\int_{\Omega}|\nabla u|^2ds+\int^{t}_0\int_{\Omega}|\nabla w|^2ds)	.
\end{equation}	
Furthermore observing that
\begin{equation*}
|\nabla w|\leq \chi_w e^{\chi_w v}|\nabla v|b+e^{\chi_w v}|\nabla b|	
\end{equation*}
as well as \begin{equation*}
|\nabla u|\leq \chi_u e^{\chi_u v}|\nabla v|a+e^{\chi_u v}|\nabla a|,	
\end{equation*}
we  conclude from \eqref{4.20} that there exists $c_1>0$ such that
\begin{equation}\label{4.21}
\begin{array}{rl}
\displaystyle\int_{\Omega}|\nabla v(\cdot,t)|^2
&\leq
c_1e^{-2\gamma t}+c_1 te^{-2\gamma t}
\displaystyle\int^{t}_0\int_{\Omega}(|\nabla b|^2+|\nabla a|^2 +|\nabla v|^2)ds\\[3mm]
&\leq
c_1e^{-2\gamma t}+c_1 te^{-2\gamma t}
\displaystyle\int^{\infty}_0\int_{\Omega}(|\nabla b|^2+|\nabla a|^2 +|\nabla v|^2)ds
\end{array}
\end{equation}	
and thus
$$
\displaystyle\int_{\Omega}|\nabla v(\cdot,t)|^2\leq c_2(t+1)e^{-2\gamma t}
$$
with some $c_2>0$,
thanks to \eqref{4.17}, \eqref{4.12} and  \eqref{4.16}.
\end{proof}
On the basis of smoothing estimates for the Neumann heat semigroup on $\Omega$, and decay information provided by Lemma \ref{Lemma46},
we can make sure that $u-1$ decays  exponentially with respect to  $L^p(\Omega)$-norm.

\begin{lemma}\label{Lemma47}
Suppose the condition in Lemma \ref{Lemma41} holds, then there exists $\eta_1>0$ such that for every $p\geq 2$,
\begin{equation}\label{4.22}
\|u(\cdot,t)-1\|_{L^p(\Omega)}\leq C(p) e^{-\eta_1 t}
\end{equation}	
with some $C(p)>0$ for all $t>0$.
\end{lemma}
\begin{proof}
Testing the first equation in \eqref{1.3} by $u-1$ and  integrating by parts, we have
\begin{equation*}
\begin{array}{rl}
&\displaystyle\frac{d}{dt}\int_{\Omega}(u-1)^2+2D_u\int_{\Omega}|\nabla u|^2+2\mu_u\int_{\Omega}u(u-1)^2\\[2mm]
=&2\xi_u\displaystyle\int_{\Omega}u\nabla v\cdot\nabla u-2\rho\int_{\Omega}(u-1)uz.
\end{array}
\end{equation*}	
We thereupon make use of Lemma \ref{Lemma42}, Lemma \ref{Lemma35} along with the Young inequality to  get
\begin{equation}\label{4.23}
\begin{array}{rl}
&\displaystyle\frac{d}{dt}\int_{\Omega}(u-1)^2+D_u\int_{\Omega}|\nabla u|^2+2\mu_u\gamma \int_{\Omega}(u-1)^2\\
\leq &\displaystyle \frac{\xi^2_u}{D_u}\int_{\Omega}u^2|\nabla v|^2+2\rho\int_{\Omega}uz\\
\leq & c_1\displaystyle \int_{\Omega}|\nabla v|^2+c_1\int_{\Omega}z
\end{array}
\end{equation}	
with some $ c_1>0$.

According to Lemma \ref{Lemma46} and Lemma \ref{Lemma23},  \eqref{4.23} implies that
\begin{equation}\label{4.24}
\int_{\Omega}(u-1)^2\leq c_2 e^{-\eta_1 t}
\end{equation}	
with $\eta_1:=\min\{2\mu_u\gamma,2\gamma,\delta\}$ and  $c_2>0$ for all $t>0$.

Recalling known smoothing estimates for the Neumann heat semigroup on $\Omega\subset \mathbb{R}^2$ (\cite{WinklerJDE}),
 there exist $c_3=c_3(p,q)>0$, $c_4=c_4(p,q)>0$ fulfilling
\begin{equation}\label{4.25}
\left \| e^{\sigma D_u\Delta } \varphi\right \|_{L^{p}(\Omega)}
\leq c_3 \sigma^{-(\frac 1q-\frac 1p)}
 \| \varphi  \|_{L^q(\Omega)}
\end{equation}
for each $\varphi\in  C^0(\Omega)$, and
 for all $\varphi\in \left ( L^{q}\left ( \Omega  \right ) \right )^{2}$,
\begin{equation}\label{4.26}
\left \| e^{\sigma D_u\Delta }\nabla\cdot \varphi\right \|_{L^{p}\left ( \Omega  \right )}\leq c_{4} ( 1+\sigma^{-\frac{1}{2}-( \frac{1}{q}-\frac{1}{p}) } )e^{-\lambda _{1}\sigma}\left \|  \varphi \right \|_{L^{q}\left ( \Omega  \right )}
\end{equation}
with $ \lambda _{1}> 0$  the first nonzero eigenvalue of $-\Delta$ in $\Omega $ under the Neumann boundary condition.

Relying on a variation-of-constants representation of $u$ related to the the first equation in \eqref{1.3}, we utilize \eqref{4.25} and \eqref{4.26} to infer that
\begin{equation}\label{4.27}
\begin{array}{rl}
&\|(u-1)(\cdot,t)\|_{L^p(\Omega)}\\
\leq &
\|e^{t(D_u\Delta-\delta )}(u_0-1)\|_{L^p(\Omega)}+
\xi_u\displaystyle \int^{t}_0\|e^{(t-s)(D_u\Delta-\delta)} \nabla\cdot(u\nabla v) \|_{L^p(\Omega)}ds\\
&+
\displaystyle\int^{t}_0\|e^{(t-s)(D_u\Delta-\delta)}((\mu_u u-\delta)(1-u)-\rho uz)\|_{L^p(\Omega)}ds	\\
 \leq & \displaystyle e^{-\delta t}\|u_0-1\|_{L^{p}(\Omega)}+
  c_5(p)
\int^{t}_0(1+(t-s)^{-1+\frac1p})e^{-(\delta+\lambda_1)(t-s)}\|\nabla v(\cdot,s)\|_{L^{2}(\Omega)}ds\\
&
+c_5(p)\displaystyle\int^{t}_0(1+(t-s)^{-\frac12+\frac1p})e^{-\delta(t-s)}  \|(u-1)(\cdot,s)\|_{L^2(\Omega)}ds\\
&
+c_5(p)\displaystyle\int^{t}_0(1+(t-s)^{-1+\frac1p})e^{-\delta(t-s)}  \|z(\cdot,s)\|_{L^1(\Omega)}ds
\end{array}
\end{equation}
for  some $c_5(p)>0$.
Therefore by \eqref{4.24}, \eqref{4.19}, \eqref{2.8}  and thanks to the fact that for
 $\alpha\in(0,1) $  $\gamma_1$ and  $\delta_1$ positive constants with $ \gamma_1 \neq \delta_1$, there exists $c_6> 0$ such that
$$\int_{0}^{t} ( 1+( t-s  ) ^{-\alpha})e^{-\gamma_1 s}e^{-\delta_1 ( t-s )}ds
\leq c_{6} e^{-min\left \{ \gamma_1 ,\delta_1  \right \}t},
$$
\eqref{4.22} readily results from  \eqref{4.27} with $\eta_1=\frac 12 \min\{ \gamma, \mu_u \gamma,
\frac{1-\beta}2, \frac {\delta_z}2\}$ and some $C(p)>0$.
\end{proof}

At this position, due to the fact that the integrability exponent in \eqref{4.19} does not exceed the considered spatial dimension $n=2$, the uniform decay of $w$ is not achieved herein, however a somewhat optimal decay  rate thereof  with respect to  $L^p(\Omega)$ may be derived by the argument similar to that in Lemma \ref{Lemma47} instead of the simple interpolation.  The desired result can be stated below and the corresponding proof is omitted herein.

\begin{lemma}\label{Lemma48}
Let the condition in Lemma \ref{Lemma41}  hold.  Then there exists $\varrho_1>0$ such that for every $p\geq 2$,
\begin{equation}\label{4.28}
\|w(\cdot,t)\|_{L^p(\Omega)}\leq C(p) e^{-\varrho_1 t}
\end{equation}	
with some $C(p)>0$ for all $t>0$.
\end{lemma}

Next we proceed to establish the convergence properties in \eqref{1.6}--\eqref{1.7} stated in Theorem 1.1, which are beyond that in Lemma \ref{Lemma47} and Lemma \ref{Lemma48}. To this end, thanks to   Lemma \ref{Lemma47}, Lemma \ref{Lemma48}, Lemma \ref{Lemma41},  Lemma \ref{Lemma42} and Lemma \ref{Lemma43},  we turn to make sure that  $\int_{\Omega}|\nabla v|^4$   decays exponentially, which results from a series of testing procedures.

\begin{lemma}\label{Lemma49}
Let conditions in Theorem \ref{Theorem 1.1} hold. Then there exist $\eta_2>0$ and  $C>0$
such that
\begin{equation}\label{4.29}
\int_{\Omega}|\nabla v|^4\leq C e^{-\eta_2 t}~~\hbox{for all $t>0$.}	
\end{equation}
\end{lemma}
\begin{proof}
 Testing the identity
\begin{equation*}
a_t=D_u \triangle a + D_u \nabla v \cdot \nabla a + f(x,t), \quad x\in\Omega, \quad t > 0
\end{equation*}	
with $
f(x,t)=\mu_u a(1-u)-\displaystyle\rho a z+\chi_u a(\alpha_u u
+ \alpha_w w)v$ by $-\triangle a$, and using Young's inequality, we get
\begin{equation}\label{4.30}
\begin{array}{rl}
\displaystyle \frac d {dt}\int_\Omega |\nabla a|^2 +2D_u\displaystyle \int_\Omega|\triangle a|^2 =&-2D_u \chi_u \displaystyle \int_\Omega  (\nabla a\cdot \nabla v)\triangle a-
2\int_\Omega f\triangle a\\
\leq& D_u\displaystyle \int_\Omega|\triangle a|^2 +2D_u\chi_u^2 \int_\Omega  |\nabla a|^2 |\nabla v|^2+\frac 2{D_u}\int_\Omega  |f|^2.
\end{array}
\end{equation}
Note that  by the Gagliardo--Nirenberg type interpolation with standard elliptic
regularity theory and Poincar\'{e}'s inequality,  one can find constants $c_1> 0$ and  $c_2> 0$ such that for all $\varphi\in W^{2,2}(\Omega)$ with $\frac{\partial \varphi}{\partial\nu }=0 $ on $\partial\Omega$,
\begin{equation*} \|\nabla \varphi\|^4_ {L^4(\Omega)}\leq c_1 \| \Delta \varphi\|^2_{L^2(\Omega)}\| \varphi\|^2_{L^\infty(\Omega)}
\end{equation*}
and $$\|\nabla \varphi\|^2_ {L^2(\Omega)}\leq c_2 \| \Delta \varphi\|^2_{L^2(\Omega)}
$$
(see Lemma A.1 and A.3 in \cite{Fuest}). Hence thanks to Lemma \ref{Lemma35}, we can pick $c_3>0$ such that
$$a(x,t)\leq c_3, \quad  b(x,t)\leq c_3
  $$
for all $x \in \Omega$ and $t > 0$, and thereby have
\begin{equation}\label{4.31}
 \|\nabla a\|^4_ {L^4(\Omega)}\leq c_1c_3^2 \| \Delta a\|^2_{L^2(\Omega)},\quad
\|\nabla a\|^2_ {L^2(\Omega)}\leq c_2 \| \Delta a\|^2_{L^2(\Omega)}
\end{equation}
as well as  \begin{equation} \label{4.32}
\|\nabla b\|^4_ {L^4(\Omega)}\leq c_1c_3^2 \| \Delta b\|^2_{L^2(\Omega)},\quad
\|\nabla b\|^2_ {L^2(\Omega)}\leq c_2 \| \Delta b\|^2_{L^2(\Omega)}
\end{equation}
Combining \eqref{4.31} with  \eqref{4.30}, the Young inequality  shows
 that for every $\varepsilon>0$
 \begin{equation}\label{4.33}
\begin{array}{rl}
&\displaystyle \frac d {dt}\|\nabla a\|^2_ {L^2(\Omega)} + \frac{D_u}{2c_2}\|\nabla a\|^2_ {L^2(\Omega)}+
D_u\displaystyle \|\triangle a\|^2_ {L^2(\Omega)} \\[3mm]
\leq & \varepsilon \|\nabla a\|^4_ {L^4(\Omega)} +  \displaystyle\frac{D_u}{2c_2}\|\nabla a\|^2_ {L^2(\Omega)}
+ \displaystyle\frac {D_u^2\chi_u^4}{\varepsilon} \|\nabla v\|^4_ {L^4(\Omega)}
+\frac 2{D_u}\|f\|_{L^2(\Omega)}^2\\
\leq&  \varepsilon c_1  c_3^2 \displaystyle \| \Delta a\|^2_{L^2(\Omega)}+\displaystyle  \frac {D_u}2\|\triangle a\|^2_ {L^2(\Omega)}
+ \displaystyle\frac {D_u^2\chi_u^4}{\varepsilon} \|\nabla v\|^4_ {L^4(\Omega)}
+\frac 2{D_u}\|f\|_{L^2(\Omega)}^2,
\end{array}
\end{equation}
which along with the choice of $\varepsilon=\frac{D_u}{4c_1c_3^2}$ implies that

 \begin{equation}\label{4.34}
\begin{array}{rl}
&\displaystyle \frac d {dt}\|\nabla a\|^2_ {L^2(\Omega)} + \frac{D_u}{2c_2}\|\nabla a\|^2_ {L^2(\Omega)}+
\displaystyle\frac {D_u}4 \|\triangle a\|^2_ {L^2(\Omega)} \\[3mm]
\leq&  \displaystyle 4D_u\chi_u^4  c_1  c_3^2 \|\nabla v\|^4_ {L^4(\Omega)}
+\frac 2{D_u}\|f\|_{L^2(\Omega)}^2.
\end{array}
\end{equation}
Likely, we can get
\begin{equation}\label{4.35}
\begin{array}{rl}
&\displaystyle \frac d {dt}\|\nabla b\|^2_ {L^2(\Omega)} + \frac{D_w}{2c_2}\|\nabla b\|^2_ {L^2(\Omega)}+
\displaystyle\frac {D_w}4 \|\triangle b\|^2_ {L^2(\Omega)} \\[3mm]
\leq&  \displaystyle 4D_w\chi_w^4  c_1  c_3^2 \|\nabla v\|^4_ {L^4(\Omega)}
+\frac 2{D_w}\|g\|_{L^2(\Omega)}^2.
\end{array}
\end{equation}
with $
g(x,t)=- b+\displaystyle\rho ue^{-\chi_w v}z
+
\chi_w b(\alpha_u u+ \alpha_w w)v$.

Now in order to appropriately compensate the first summand on right-hand side of \eqref{4.34} and  \eqref{4.35},  we use
the third equation in \eqref{2.2} to see that
\begin{equation}\label{4.36}
\begin{array}{rl}
&\displaystyle \frac 14 \frac d {dt}\int_\Omega |\nabla v|^4\\
  =& -\displaystyle  \int_\Omega|\nabla v|^2 \nabla v\cdot\nabla v_t\\
  =&-\alpha_u \displaystyle \int_\Omega   a(v+\chi_u)e^{\chi_u v}|\nabla v|^4-
\alpha_u \displaystyle \int_\Omega   ve^{\chi_u v}|\nabla v|^2 \nabla v\cdot\nabla a\\
&-\alpha_w \displaystyle \int_\Omega   b(v+\chi_w)e^{\chi_w v}|\nabla v|^4-
\alpha_w \displaystyle \int_\Omega   ve^{\chi_w v}|\nabla v|^2 \nabla v\cdot\nabla b.
\end{array}
\end{equation}
Here recalling the uniform positivity of $a$ stated in Lemma \ref{Lemma42}, we can pick $c_4>0$ fulfilling
$$
\alpha_u \displaystyle \int_\Omega   a(v+\chi_u)e^{\chi_u v}|\nabla v|^4
\geq \alpha_u \chi_uc_4 \displaystyle \int_\Omega  |\nabla v|^4
$$
and thus  infer by the Young inequality and  Lemma \ref{Lemma22} that for all $t>t_0$
\begin{equation}\label{4.37}
\begin{array}{rl}
&\displaystyle  \frac d {dt}\int_\Omega |\nabla v|^4 +c_5\int_\Omega |\nabla v|^4
\\
  \leq & c_6\|v(\cdot,t_0)\|_{L^\infty(\Omega)}\displaystyle \int_\Omega  (|\nabla a|^4+ |\nabla b|^4)
\end{array}
\end{equation}
with constants $c_5>0$, $c_6>0$.

Now if we write $d_1:=\frac{8c_1c_3^2(  D_u\chi_u^4+ D_w\chi_w^4)} {c_5}$, combining \eqref{4.37}, \eqref{4.34} with  \eqref{4.35} yields
 \begin{equation}\label{4.38}
\begin{array}{rl}
&\displaystyle \frac d {dt}\left(\|\nabla a\|^2_ {L^2(\Omega)} + \|\nabla b\|^2_ {L^2(\Omega)}+d_1\|\nabla v\|_ {L^4(\Omega)}^4\right)\\[3mm]
&+
\displaystyle \frac{D_u}{2c_2}\|\nabla a\|^2_ {L^2(\Omega)}+
\displaystyle\frac {D_u}4 \|\triangle a\|^2_ {L^2(\Omega)}
+ \frac{D_w}{2c_2}\|\nabla b\|^2_ {L^2(\Omega)}+
\displaystyle\frac {D_w}4 \|\triangle b\|^2_ {L^2(\Omega)} +\frac{c_5 d_1}2\|\nabla v\|_ {L^4(\Omega)}^4
\\[3mm]
\leq& c_6d_1\|v(\cdot,t_0)\|_{L^\infty(\Omega)} (\|\nabla a\|_ {L^4(\Omega)}^4 +\|\nabla b\|_ {L^4(\Omega)}^4)
+\displaystyle\frac 2{D_u}\|f\|_{L^2(\Omega)}^2+\frac 2{D_w}\|g\|_{L^2(\Omega)}^2,
\end{array}
\end{equation}
which together with \eqref{4.31}, \eqref{4.32} and Lemma \ref{Lemma43} entails that there exists $t_1>1$ suitably large such that
for all $t>t_1$, \begin{equation}\label{4.39}
\begin{array}{rl}
&\displaystyle\frac d {dt}\left(\|\nabla a\|^2_ {L^2(\Omega)} + \|\nabla b\|^2_ {L^2(\Omega)}+d_1\|\nabla v\|_ {L^4(\Omega)}^4\right)\\[3mm]
&+
\displaystyle\frac{D_u}{2c_2}\|\nabla a\|^2_ {L^2(\Omega)}+ \frac{D_w}{2c_2}\|\nabla b\|^2_ {L^2(\Omega)}+\frac{c_5 d_1}2\|\nabla v\|_ {L^4(\Omega)}^4
\\[3mm]
\leq&  \displaystyle\frac 2{D_u}\|f\|_{L^2(\Omega)}^2+\frac 2{D_w}\|g\|_{L^2(\Omega)}^2.
\end{array}
\end{equation}
 Due to Lemma \ref{Lemma35}, we see that
 $$|f(x,t)|^2+|g(x,t)|^2\leq c_7(|u(x,t)-1|^2+|w(x,t)|^2+ |z(x,t)|^2+|v(x,t)|^2)
 $$
  with some $c_7>0$,  and thus  there exist $\eta_3>0$ and  $c_8>0$
such that
\begin{equation}\label{4.40}
\int_{\Omega}|f(x,t)|^2+ |g(x,t)|^2 \leq c_8 e^{-\eta_3 t}~~\hbox{for all $t>t_1,$}	
\end{equation}
thanks to   Lemma \ref{Lemma47}, Lemma \ref{Lemma48}, Lemma \ref{Lemma41}  and Lemma \ref{Lemma43}. Therefore \eqref{4.29} readily results from
\eqref{4.40} and \eqref{4.39}.
 \end{proof}

At this position, as an application of known smoothing estimates for the Neumann heat semigroup, the latter readily turns to
 the exponential decay property of $u-1$ as well as $w$ with respect to $L^\infty(\Omega)$-norm.
\begin{lemma}\label{Lemma410}
Let the conditions in Theorem \ref{Theorem 1.1} hold. Then there exist $\eta,\varrho>0$ and  $C>0$ fulfilling
\begin{equation}\label{4.41}
\|u(\cdot,t)-1\|_{L^\infty(\Omega)}\leq C e^{-\eta t}
\end{equation}	
as well as   \begin{equation}\label{4.42}
\|w(\cdot,t)\|_{L^{\infty}(\Omega)}\leq C e^{-\varrho t}	
\end{equation}
  for all  $t>0$.
\end{lemma}
\begin{proof}
 Since the proof is similar to that of  Lemma \ref{Lemma47}, we only give a short proof of  \eqref{4.41}.
In view to known smoothing estimates for the Neumann heat semigroup on $\Omega\subset \mathbb{R}^2$ (\cite{WinklerJDE}),
 there exist $c_1>0$, $c_2>0$ fulfilling
\begin{equation}\label{4.43}
\left \| e^{\sigma D_u\Delta } \varphi\right \|_{L^{\infty}(\Omega)}
\leq c_1 \sigma^{-\frac 12}
 \| \varphi  \|_{L^2(\Omega)}
\end{equation}
for each $\varphi\in  C^0(\Omega)$, and
 for all $\varphi\in \left ( L^{4}\left ( \Omega  \right ) \right )^{2}$,
\begin{equation}\label{4.44}
\left \| e^{\sigma D_u\Delta }\nabla\cdot \varphi\right \|_{L^{\infty}\left ( \Omega  \right )}\leq c_2 ( 1+\sigma^{-\frac34} )e^{-\lambda_{1}\sigma}\left \|  \varphi \right \|_{L^{4}\left ( \Omega  \right )}
\end{equation}
with $ \lambda _{1}> 0$  the first nonzero eigenvalue of $-\Delta$ in $\Omega $ under the Neumann boundary condition.

According to the variation-of-constants representation of $u$ related to the the first equation in \eqref{1.3}, we utilize \eqref{4.43} and \eqref{4.44} to infer that
\begin{equation}\label{4.45}
\begin{array}{rl}
&\|(u-1)(\cdot,t)\|_{L^\infty(\Omega)}\\
\leq &
\|e^{t(D_u\Delta-1 )}(u_0-1)\|_{L^\infty(\Omega)}+
\xi_u\displaystyle \int^{t}_0\|e^{(t-s)(D_u\Delta-1)} \nabla\cdot(u\nabla v) \|_{L^\infty(\Omega)}ds\\
&+
\displaystyle\int^{t}_0\|e^{(t-s)(D_u\Delta-1)}((\mu_u u-1)(1-u)-\rho uz)\|_{L^\infty(\Omega)}ds	\\
 \leq &
 \displaystyle e^{-t}\|u_0-1\|_{L^{\infty}(\Omega)}+
  c_3
\int^{t}_0(1+(t-s)^{-\frac34})e^{-(1+\lambda_1)(t-s)}\|\nabla v(\cdot,s)\|_{L^{4}(\Omega)}ds\\
&
+c_3\displaystyle\int^{t}_0(1+(t-s)^{-\frac12})e^{-(t-s)}
 (\|(u-1)(\cdot,s)\|_{L^2(\Omega)}
+ \|z(\cdot,s)\|_{L^2(\Omega)})ds
\end{array}
\end{equation}
with some $c_3>0$.
This  readily  establish \eqref{4.41} with appropriate $\eta>0$ in view of  \eqref{4.29}, \eqref{4.22} and \eqref{4.1}.
\end{proof}

Thereby  our main result has essentially been proved already.

\bf{Proof of Theorem 1.1.}\rm \, The statement on global boundedness of classical solutions
has been asserted by Lemma \ref{Lemma35}.  The convergence properties in \eqref{1.6}--\eqref{1.9} are
precisely established by  Lemma \ref{Lemma410},  Lemma \ref{Lemma41} and Lemma \ref{Lemma43}, respectively.

\section{Acknowledgments}
This work is partially supported by NSFC (No.12071030).


\begin{thebibliography}{00}
\bibitem{ARD-MB}T. Alzahrani,~R. Eftimie,~D. Trucu,~\emph{Multiscale modelling of cancer response to oncolytic viral therapy},~Math.~Bioci.~310(2019),~76--95.

\bibitem{ACNST} A. R. Anderson, M. A. J. Chaplain, E. L. Newman, R. J. C. Steele, A. M. Thompson,
~\emph{Mathematical modelling of tumour invasion and metastasis},  J. Theor. Med. 2(2000), 129--154.

\bibitem{Cao} X. Cao,  \emph{Boundedness in a three-dimensional chemotaxis--haptotaxis model},  Z. Angew. Math. Phys. 67(2016), 67:11.



\bibitem{Chen}Z. Chen,~\emph{Dampening effect of logistic source in a two-dimensional haptotaxis system with nonlinear zero-order interaction},~J.~Math.~Anal.~Appl.~492(2020), 124435.


\bibitem{CL}
M.A.J. Chaplain,  G. Lolas,~\emph{Mathematical modelling of cancer cell invasion of tissue: the
role of the urokinase plasminogen activation system},  Math. Mod. Meth. Appl. Sci. 18(2005), 1685--1734.


\bibitem{FFH}
M.A. Fontelos, A. Friedman, B. Hu, \emph{Mathematical analysis of a model for the initiation of angiogenesis}, SIAM J.
Math. Anal. 33(2002), 1330--1355.

\bibitem{Fuest}
M. Fuest, \emph{Global solutions near homogeneous steady states in a multi-dimensional population model
 with both predator-and prey-taxis}, SIAM J. Math. Anal. 52(2020), 5863--5891.


\bibitem{FIT}
H. Fukuhara, Y. Ino, T. Todo, \emph{Oncolytic virus therapy: A new era of cancer treatment
at dawn}, Cancer Sci. 107(2016), 1373--1379.

\bibitem{GPKLK}
 S. Gujar, J. G. Pol, Y. Kim, P. W. Lee, G. Kroemer, \emph{Antitumor benefits of antiviral immunity: An underappreciated aspect of oncolytic virotherapies}, Trends Immunol. 39(2018), 209--221.


\bibitem{GK}
I. Ganly, D. Kirn, et al.,  ~\emph{A phase I study of Onyx-015, an E1B-attenuated adenovirus, administered intratumorally to
patients with with recurrent head and neck cancer}, Clinical Cancer Res. 6(2000), 798--806.

\bibitem{Jin} C. Jin,~\emph{Global classical solutions and convergence to
a mathematical model for cancer cells invasion and metastatic spread},  J. Diff. Equations 269(2020), 3987--4021.

\bibitem{JinTian}
H.Y. Jin, T.~Xiang, ~\emph{Negligibility of haptotaxis effect in a chemotaxis--haptotaxis model},
 Math. Mod. Meth. Appl. Sci. 31(7)(2021), 1373--1417.

\bibitem{K}
N. L. Komarova, ~\emph{Viral reproductive strategies: how can lytic viruses be evolutionarily
competitive?}, J. Theor. Biol. 249(2007), 766--784.

\bibitem{LiL}
Y. Li, J. Lankeit,  \emph{Boundedness in a chemotaxis--haptotaxis model with nonlinear diffusion}, Nonlinearity 29(2016), 1564--1595.

\bibitem{LiWang}
J. Li, Y. Wang, \emph{Boundedness in a haptotactic cross--diffusion system modeling oncolytic
virotherapy}. J. Diff. Equation  270(2021), 94--113.


\bibitem{NG}
 J.~Nemunaitis, I. Ganly, et al.,   ~\emph{Selective replication and oncolysis in p53 mutant tumors with ONYX-015, an E1B-55kD gene-deleted adenovirus,
in patients with advanced head and neck cancer: a phase II trial},  Cancer Res. 60(2000), 6359--6366.

  \bibitem{PW3MAS}
P. Y. ~H. Pang, Y. Wang,~\emph{Global boundedness of solutions to a chemotaxis--haptotaxis model with tissue remodeling},
Math. Models Methods Appl. Sci. 28(2018),  2211--2235.

  \bibitem{PW3MAS2}
P. Y. ~H. Pang, Y. Wang,~\emph{
Asymptotic behavior of solutions to a tumor angiogenesis model with chemotaxis--haptotaxis}, Math. Models Methods Appl. Sci.
29(2019), 1387--1412.


\bibitem{PZS}
J. Pr\"{u}ss, R. Zacher, R. Schnaubelt,  ~\emph{Global asymptotic stability of equilibria in models
for virus dynamics},   Math. Model. Nat. Phenom. 3(2008), 126--142.

\bibitem{RL-MMA} G. Ren, ~B. Liu,
~\emph{Global classical solvability in a three-dimensional haptotaxis system modelong oncolytic virotherapy},
~Math.~Methods Appl.~Sci.~44(2021),~9275--9291.


\bibitem{SSW}
C. Stinner, C.~Surulescu, M. ~Winkler, ~\emph{Global weak solutions in a PDE-ODE system
modeling multiscale cancer cell invasion}, SIAM J.  Math.  Anal.  46(2014), 1969--2007.


\bibitem{TW-JDE2020}
Y. Tao, M. Winkler,~\emph{Global classical solutions to a doubly haptotactic
cross--diffusion system modeling oncolytic virotherapy},~J. Diff. Equations  268(2020), 4973--4997.


\bibitem{TW-PRSE}Y.Tao,~M. Winkler,~\emph{Asymptotic stability of spatial homogeneity in a haptotaxis model for oncolytic virotherapy},~Proc.~Roy.~Soc.~Edinburgh Sect. A  ~(2021),~1--21. DOI: https://doi.org/10.1017/prm.2020.97


\bibitem{TW-DCDS}Y.Tao,~M. Winkler,~\emph{Critical mass for infinite-time blow-up in a haptotaxis system with nonlinear zero-order interaction},~Discrete Contin.~Dyn.~Syst. A  ~41(2021),~439--454.

\bibitem{TW-EJAM}Y.Tao,~M.Winkler,~\emph{A critical virus production rate for efficiency of oncolytic virotherapy},~European J.~Appl.~Math.~32(2021),~301--316.

\bibitem{TW-NA}Y.Tao,~M. Winkler,~\emph{A critical virus production rate for blow-up suppression in a haptotaxis model for oncolytic virotherapy},~Nonlinear Anal. 198(2020), 111870.

\bibitem{TaoSiam} Y. Tao, M. Winkler, ~\emph{Large time behavior in a multidimensional chemotaxis--haptotaxis model with slow
signal diffusion}, SIAM J. Math. Anal. 47(2015), 4229--4250.

\bibitem{TW-JDE}Y. Tao,~M. Winkler,~\emph{Energy-type estimates and global solvability in a two-dimensional chemotaxis-haptotaxis model with remodeling of non-diffusible attractant},~J. Diff. Equations 257(2014),  ~784--815.


\bibitem{WalkerW}
C. Walker,  G.F. Webb,  \emph{Global existence of classical solutions for a haptotaxis model}, SIAM
J. Math. Anal. 38(2007), 1694--1713.

\bibitem{WYF}
Y.~Wang, \emph{Boundedness in the higher-dimensional chemotaxis--haptotaxis model with nonlinear
diffusion}, J. Diff. Equations 260(2016), 1975--1989.


\bibitem{WinklerJDE}M. Winkler, ~\emph{Aggregation vs. Global diffusive behavior in the higher-dimensional Keller--Segel model}, J. Diff. Equations 12(2010), 2889--2905.

\bibitem{ZhengK}
J. Zheng,  Y. Ke, \emph{Large time behavior of solutions to a fully parabolic chemotaxis--haptotaxis model in N dimensions}, J. Diff. Equations  266(2019), 1969--2018.

\bibitem{ZCU}
 A. Zhigun,  C. Surulescu, A. Uatay, \emph{Global existence for a degenerate haptotaxis model of cancer invasion},
 Z. Angew. Math. Phys. 6(2016), 67:146.

\end{thebibliography}
\end{document}